\documentclass{amsart}
\usepackage{amsmath}
\usepackage{paralist}
\usepackage{url}
\usepackage{mathrsfs}
\usepackage{color}
\usepackage{graphicx}
\usepackage{multirow}
\usepackage{mathtools}
\usepackage{amssymb}
\usepackage{booktabs}
\usepackage{makecell}
\usepackage{pgfplots}
\pgfplotsset{compat=newest}
\usepackage{mathrsfs}
\usetikzlibrary{arrows}
\usepackage{xspace}
\usepackage{comment}
\usepackage{subcaption}

\newcommand{\re}{\mathbb{R}}

\newcommand{\N}{\mathbb{N}}

\newcommand{\lmd}{\lambda}

\newcommand{\eps}{\epsilon}

\newcommand{\dt}{\delta}

\def\rank{\mbox{rank}}

\newcommand{\reff}[1]{(\ref{#1})}

\newcommand{\mc}[1]{\mathcal{#1}}

\newcommand{\st}{\mathit{s.t.}}

\newcommand{\qmod}[1]{\mbox{QM}[#1]}

\newcommand{\RN}[1]{%
  \textup{\uppercase\expandafter{\romannumeral#1}}%
}

\newcommand{\bdes}{\begin{description}}
	\newcommand{\edes}{\end{description}}
\newcommand{\bal}{\begin{align}}
\newcommand{\eal}{\end{align}}
\newcommand{\bnum}{\begin{enumerate}}
	\newcommand{\enum}{\end{enumerate}}
\newcommand{\bit}{\begin{itemize}}
	\newcommand{\eit}{\end{itemize}}
\newcommand{\bea}{\begin{eqnarray}}
\newcommand{\eea}{\end{eqnarray}}
\newcommand{\be}{\begin{equation}}
\newcommand{\ee}{\end{equation}}
\newcommand{\baray}{\begin{array}}
	\newcommand{\earay}{\end{array}}
\newcommand{\bsry}{\begin{subarray}}
	\newcommand{\esry}{\end{subarray}}
\newcommand{\bca}{\begin{cases}}
	\newcommand{\eca}{\end{cases}}
\newcommand{\bcen}{\begin{center}}
	\newcommand{\ecen}{\end{center}}
\newcommand{\bbm}{\begin{bmatrix}}
	\newcommand{\ebm}{\end{bmatrix}}
\newcommand{\bmx}{\begin{matrix}}
	\newcommand{\emx}{\end{matrix}}
\newcommand{\bpm}{\begin{bmatrix}}
	\newcommand{\epm}{\end{bmatrix}}
\newcommand{\btab}{\begin{tabular}}
	\newcommand{\etab}{\end{tabular}}

\newcommand{\MATLAB}{\textsc{Matlab}\xspace}

\newtheorem{theorem}{Theorem}[section]

\newtheorem{prop}[theorem]{Proposition}

\newtheorem{corollary}[theorem]{Corollary}

\theoremstyle{definition}
\newtheorem{example}[theorem]{Example}

\newtheorem{alg}[theorem]{Algorithm}

\newtheorem{remark}[theorem]{Remark}

\theoremstyle{plain}

\newcommand{\Rni}{\mathbb{R}^{n_{i}}}

\newcommand{\xpi}{x_i}

\newcommand{\xmi}{x_{-i}}

\newcommand{\fpi}{f_{i}}

\newcommand{\ddd}{,\ldots,}

\numberwithin{equation}{section}

\begin{document}

\title[GNEPs with Quasi-linear constraints]
{Generalized Nash equilibrium problems with Quasi-linear constraints}

\author[Jiyoung Choi]{Jiyoung~Choi}
\author[Jiawang Nie]{Jiawang~Nie}
\author[Xindong Tang]{Xindong~Tang}
\author[Suhan Zhong]{Suhan~Zhong}

\address{Jiyoung Choi, Jiawang Nie, Department of Mathematics, University of California San Diego, 9500 Gilman Drive, La Jolla, CA, USA, 92093.}
\email{jichoi@ucsd.edu, njw@math.ucsd.edu}

\address{Xindong Tang, Department of Mathematics,
Hong Kong Baptist University,
Kowloon Tong, Kowloon, Hong Kong.}
\email{xdtang@hkbu.edu.hk}

\address{Suhan Zhong, School of Mathematical Sciences,
	Shanghai Jiao Tong University, Shanghai, China, 200240.}
\email{suzhong@sjtu.edu.cn}

\date{}

\begin{abstract}
We study generalized Nash equilibrium problems (GNEPs)
such that objectives are polynomial functions, and each player's constraints are linear in their own strategy.
For such GNEPs, the KKT sets can be represented as unions of simpler sets
by Carath\'{e}odory's theorem. We give a convenient representation for KKT sets using partial Lagrange multiplier expressions.
This produces a set of
branch polynomial optimization problems,
which can be efficiently solved by Moment-SOS relaxations.
By doing this, we can compute all
generalized Nash equilibria or detect their nonexistence.
This method may not be very scalable to large scale GNEPs.
Numerical experiments are provided to
demonstrate the computational efficiency.
\end{abstract}

\maketitle

\section{Introduction}

The generalized Nash equilibrium problem (GNEP) is a class of games that determines strategies for a group of players so that each player's benefit cannot be improved for the given strategy of other players.
Suppose there are $N$ players, and the $i$th player's strategy is represented by
the $n_i$-dimensional real vector
$x_i\coloneqq (x_{i,1},\ldots, x_{i,n_i}) \in \re^{n_i}$.
The tuple  $x  \coloneqq  (x_1,\ldots,x_N)$
denotes the set of all players' strategies, with the total dimension
$n  \coloneqq  n_1+ \cdots + n_N .$
When the $i$th player's strategy $x_i$ is focused, for convenience, we also write
\[
x = (x_i, x_{-i}),\quad \mbox{with}\quad
x_{-i}\, \coloneqq \, (x_1,\ldots,x_{i-1},x_{i+1},\ldots,x_N).
\]
Assume the $i$th player's decision optimization problem is
\be
\label{eq:GNEP}
\mbox{F}_i(x_{-i}): \,
\left\{ \begin{array}{cl}
\min\limits_{\xpi\in \Rni}  &  \fpi(x_i,x_{-i}) \\
\st & g_i(x_i,x_{-i}) \geq 0 ,
\end{array} \right.
\ee
where $g_i: \re^n \to \re^{m_i}$ is an $m_i$-dimensional vector-valued function.
For convenience, we only consider inequality constrained GNEPs.
The discussion for GNEPs with equality constraints is quite similar.
A tuple of strategies $u = (u_1, \ldots, u_N)$ is said to be a {\it generalized Nash
equilibrium} (GNE) if each $u_i$ is a minimizer of $\mbox{F}_i(u_{-i})$.
Throughout the paper, an optimizer means a global optimizer,
unless its meaning is specified.
For each $i=1\ddd N$ and for a given $x_{-i}$, we let
$\mc{S}_i(x_{-i})$ denote the set of minimizers for $\mbox{F}_i(x_{-i})$.
Therefore, the set $\mc{S}$ of all GNEs can be written as
\be   \label{set:S}
\mc{S}\, = \, \Big \{(u_1,\ldots,u_N):\   u_i\in \mc{S}_i(u_{-i})
\,\, \mbox{for} \,\, i = 1, \ldots, N \Big \}.
\ee

In this paper, we consider a broad class of GNEPs such that
all objectives $f_i$ are polynomials in $x$,
while the optimization problem $\mbox{F}_i(x_{-i})$ has
{\it constraints that are linear in $x_i$}
(they may be polynomial in $x_{-i}$, hence we call them quasi-linear constraints).
We assume that $g_i = (g_{i,1}, \ldots, g_{i,m_i})$
is such that for each $j\in [m_i] \coloneqq \{1,\ldots, m_i\},$
it holds
\be   \label{eq:gij}
g_{i,j}(x) \,=\, a_{i,j}^Tx_i-b_{i,j}(x_{-i}),
\ee
where each $a_{i,j}$ is a constant $n_i$-dimensional real vector
and $b_{i,j}(x_{-i})$ is a scalar polynomial in $\xmi$.
For convenience, denote
\[ A_i \, \coloneqq \, \bbm a_{i,1} \ \ldots \ a_{i,m_i}\ebm^T,\quad 
b_i(x_{-i}) \, \coloneqq \,  \bbm b_{i,1}(x_{-i})\ \ldots \ b_{i,m_i}(x_{-i}) \ebm^T,\]
where the superscript $^T$ means the transpose of a matrix or vector.
Then $g_i(x) = A_ix_i - b_i(x_{-i})$ and the $i$th player's feasible set can be written as
\be   \label{eq:setXi}
X_i(\xmi) \, \coloneqq \, \left\{\xpi \in\re^{n_i}: \,  A_ix_i - b_i(x_{-i}) \,\geq\, 0 \right\}.
\ee
The entire feasible set of the GNEP is
\be\label{eq:X}
X \, \coloneqq \, \Big \{ (x_1\ddd x_N) : \,
 A_ix_i - b_i(x_{-i}) \,\geq\, 0
\,\, \text{for all} \, \,  i = 1,\ldots, N \Big \}.
\ee

The GNEP \reff{eq:GNEP} is called a \textit{Nash equilibrium problem} (NEP) if
each feasible set $X_i(x_{-i})$ is independent of $x_{-i}$.
A solution to the NEP is then called a {\it Nash equilibrium} (NE).
The GNEP is said to be {\it convex} if each
$\mbox{F}_i(x_{-i})$ is a convex optimization problem in $\xpi$
for every given $\xmi$ such that $X_i(x_{-i})\ne\emptyset$.
GNEPs were originally introduced to model economic problems.
They are now widely used in various applications, such as transportation, telecommunications, and machine learning.
We refer to \cite{Cui2021book,Facchinei2010,FacKan10,Facchinei2010book,Pang2005,Pang2011}
for applications and surveys of GNEPs.

Solving GNEPs is typically a challenging task, primarily due to the interactions among different players' strategies concerning the objectives and feasible sets.
There may be none or a finite number of distinct GNEs,
even for strictly convex NEPs \cite{Nie2020nash}.
Much earlier work exists to solve GNEPs.
Some of them apply classical nonlinear optimization methods,
such as the penalty method \cite{Ba2020,FacKan10}
and Augmented-Lagrangian method \cite{kanzow2016}.
Variational inequality and quasi-variational inequality reformulations
are also frequently used to solve GNEPs; see the work
\cite{Facchinei2010generalized,Pang2005,Schiro2013}.
The Nikaido-Isoda function type methods are proposed in
\cite{dreves2012nonsmooth,Heusinger2012}.
The ADMM-type methods for GNEPs in Hilbert spaces are introduced in \cite{Borgens2021}.
The Gauss-Seidel type methods are proposed in \cite{Facchinei2011}.
Methods based on {\it Karush-Kuhn-Tucker} (KKT) conditions
are given in \cite{dreves2011solution,facchinei2009generalized}.
Certain convexity assumptions are often needed for these methods to be guaranteed to compute a GNE.
Numerical algorithms for solving linear GNEPs are studied in \cite{dreves2016}.
It is generally quite challenging to solve nonconvex GNEPs.
As an alternative, for nonconvex GNEPs, some work aims to find quasi-NEs introduced in \cite{Cui2021book,Pang2011}.
For more detailed introductions to GNEPs, we refer to
\cite{Facchinei2010,Facchinei2010book}.

\subsection*{Contributions}
GNEPs given by polynomial or rational functions are studied in
\cite{nie2023convex,Nie2020gs,NieTangZgnep21}.
Particularly, in \cite{nie2023convex,NieTangZgnep21}, Moment-SOS
relaxation methods are proposed to find GNEs or to detect their nonexistence.
These methods require {\it Lagrange multiplier expressions} (LMEs),
which are known and convenient for some common constraints like simplices, balls, or cubes.
However, for more general constraints, LMEs are quite expensive to obtain.
In particular, for GNEPs with many linear constraints, the usage of LMEs is quite inconvenient.
In this paper, we study this kind of GNEP, which has many quasi-linear constraints.
The linear property of constraints can be used to get computationally convenient expressions for Lagrange multipliers.
This novel approach greatly improves the efficiency of solving GNEPs.

Note that $x = (x_1, \ldots, x_N)$ is a GNE if and only if
every $x_i \in \mc{S}_i(x_{-i})$. In computation,
one can relax $x_i \in \mc{S}_i(x_{-i})$ by KKT conditions.
For the $i$th player's decision problem $\mbox{F}_i(x_{-i})$,
these conditions are
\be \label{eq:KKT1}
\boxed{
\begin{array}{c}
\nabla_{x_i} f_i(x) - A_i^T \lambda_i = 0, \\
0\le\,  \lambda_{i}\,\perp\, \big( A_i x_i - b_i (x_{-i})  \big)\,  \ge 0.
\end{array}
}
\ee
In the above, $\nabla_{x_i}$ denotes the gradient in the subvector $x_i$ and
\[ \lambda_i = \bbm \lambda_{i,1} & \cdots & \lambda_{i,m_i} \ebm^T\]
is the vector of Lagrange multipliers.
The notation $\lambda_i\perp g_i$ means that $\lambda_i$ and $g_i(x)$ are perpendicular to each other.
The strategy vector $x$ is called a KKT point
if for each $i\in[N]$, there exists $\lmd_i\in\re^{m_i}$
such that $(x,\lmd_i)$ satisfies \reff{eq:KKT1}.
We remark that when
all $f_i$ and $b_i$ are generic polynomials
and each $A_i$ is a generic matrix,
the GNEP with quasi-linear constraints has finitely many KKT points.
This is shown in \cite{nie2022algebraic}.

It is usually not easy to solve \reff{eq:KKT1} directly to get a KKT point,
since there are Lagrange multiplier variables like $\lmd_i$.
To solve \reff{eq:KKT1} more efficiently,
LMEs are introduced in \cite{nie2023convex,NieTangZgnep21}.
Generally, there exists a vector function
$\tau_i(x)$ such that
\begin{equation}\label{eq:rLME1st}
\lmd_i = \tau_i(x) \, \mbox{ satisfies \reff{eq:KKT1} for every KKT point $x$. }
\end{equation}
Such $\tau_i(x)$ is called a Lagrange multiplier expression.
For $A_i \in \re^{m_i \times n_i}$, the transpose $A_i^T$ is $n_i$-by-$m_i$.
For the special case that $\rank\, A_i = m_i$,
\[
\lmd_i \, =  \, (A_i A_i^T)^{-1} A_i \nabla_{x_i} f_i(x) .
\]
However, for general cases where $\rank \, A_i < m_i$,
the above LMEs are not applicable.

When GNEPs have quasi-linear constraints,
we propose a computationally efficient way to get LMEs.
For each fixed GNE $x$, the KKT system \reff{eq:KKT1}
has a Lagrange multiplier vector $\lmd_i$
with at most $r_i \coloneqq \rank\, A_i$ nonzero entries.
This is implied by Carath\'{e}odory's theorem.
Suppose $J_i \subseteq [m_i]$ is the label set of
nonzero entries of $\lmd_i$, with the cardinality $|J_i| = r_i$.
Let $A_{i,J_i}\in\re^{r_i\times n_i}$ be the submatrix of $A_i$ whose rows are labelled by $J_i$,
and we define $\lmd_{i,J_i}, b_{i,J_i}$ respectively in a similar way.
Then \reff{eq:KKT1} implies the equation
\[
\nabla_{x_i} f_i(x) - A_{i,J_i}^T \lambda_{i,J_i} = 0.
\]
Assume $A_{i,J_i}$ is invertible (i.e., $r_i = n_i$). Then,
\[
\lambda_{i,J_i} = ( A_{i,J_i}^T )^{-1} \nabla_{x_i} f_i(x).
\]
When $r_i<n_i$, expressions of $\lambda_{i,J_i}$ are given in Subsection~\ref{sec:pLME}.
The above right-hand side is a polynomial function in $x$.
We call it a {\it partial Lagrange multiplier expression} (pLME) \cite{Nie23plme}.
Then, \reff{eq:KKT1} simplifies to
\be  \label{LME:Ji:perp}
0 \le\,  ( A_{i,J_i}^T )^{-1} \nabla_{x_i} f_i(x)  \,\perp\,
\big( A_{i,J_i} x_i - b_{i,J_i} (x_{-i})  \big)\,  \ge 0,\quad 
i=1\ddd N.
\ee
In computational practice, the label set $J_i$ is usually unknown.
However, we can enumerate all such
$J_i \subseteq [m_i]$ with $|J_i| = r_i$.
Therefore, the KKT set $\mc{K}$ can be represented as
\be  \label{set:mcK:Ji}
\mc{K} \, = \,
\bigcup_{ \substack{ J_i \subseteq [m_i], |J_i| = r_i }  } \Big\{ x\in X:  x \,\,  \mbox{satisfies} \,\, \reff{LME:Ji:perp}   \Big \} .
\ee
Using pLMEs for the KKT set, we propose a method for finding all GNEs,
for both convex and nonconvex GNEPs.
Our major contributions are:

\begin{itemize}

\item  We give a pLME representation for the KKT set as in \reff{set:mcK:Ji}.
The pLMEs are explicitly given in closed formulae.
They are computationally convenient and efficient.

\item Based on partial Lagrange multiplier expressions,
we give a method for computing GNEs for both convex and nonconvex GNEPs.
When the KKT set $\mc{K}$ is finite (this is the generic case),
we can find {\it all} GNEs or detect their nonexistence.
For non-generic GNEPs, our method is still applicable, but it may or may not get all GNEs.

\item
We remark that our method is {\it not} enumerating active constraining sets
since we only consider label sets $J_i \subseteq [m_i]$ with $|J_i| = r_i$.
The number of our enumerations depends on the gap $m_i - r_i$.
To the best of our knowledge, this is the first work to apply this approach to solve GNEPs.

\item
The proposed method is effective for solving GNEPs
with quasi-linear constraints for problems of moderate size.
However, we acknowledge that it may not be scalable to large-scale GNEPs.
It is still a computational challenge to solve large-scale GNEPs.
This is important future work.

\end{itemize}

This paper is organized as follows. Section~\ref{sec:pre} reviews notation and some basics for polynomial optimization.
Section~\ref{sec:KKTpLME} introduces partial Lagrange multiplier expressions.
In Section~\ref{sec:findGNEs}, we give algorithms for solving GNEPs.
Section~\ref{sc:pops} introduces the Moment-SOS relaxations for solving polynomial optimization.
Numerical experiments are presented in Section~\ref{sec:NumExp}.
Conclusions and some discussions are given in Section~\ref{sc:con}.

\section{Preliminaries}\label{sec:pre}
\subsection*{Notation}
The symbol $\mathbb{N}$ (resp., $\mathbb{R}$) represents
the set of nonnegative integers (resp., real numbers).
The $\mathbb{R}^n$ denotes the $n$-dimensional Euclidean space.
Let $\mathbb{R}[x]$ denote the ring of polynomials with real coefficients in $x$,
and $\mathbb{R}[x]_d$ denote its subset of polynomials
whose degrees are not greater than $d$.
For the $i$th player's strategy vector $x_i \in \mathbb{R}^{n_i}$, the $x_{i,j}$
denotes the $j$th entry of $x_i$, for $j = 1, \ldots, n_i$.
For the $i$th player's objective $f_i(x)$,  $\nabla_{x_i}f_i$
denotes its gradient with respect to $x_i$.
For an integer $n > 0$, $[n] \coloneqq \{1,2, \ldots, n\}$.
For a vector $u \in \mathbb{R}^n$, $\|u\|$ denotes the standard Euclidean norm.
The $e_i$ represents the vector of all zeros except that the $i$th entry is $1$, while $\mathbf{1}$ denotes the vector of all ones.
The symbol $\mathbf{0}_{n_1 \times n_2}$ stands for
the zero matrix of dimension $n_1 \times n_2$,
and the subscript may be omitted if the dimension is clear in the context.
For a set $T$, $|T|$ denotes its cardinality.
For $\alpha = (\alpha_1, \ldots, \alpha_n) \in \mathbb{N}^n$,
let $|\alpha| \coloneqq \alpha_1 + \cdots + \alpha_n$, $x^{\alpha} \coloneqq x_1^{\alpha_1} \cdots x_n^{\alpha_n}$
and denote the power set $\mathbb{N}^{n}_d \coloneqq \{\alpha \in \N^n: |\alpha|\leq d\}$.
The column vector of all monomials in $x$ and of degrees up to $d$ is denoted as
\be  \label{[x]d}
[x]_d \coloneqq
[ 1 \quad x_1 \quad \cdots \quad x_n \quad x_1^2 \quad x_1x_2 \quad \cdots \quad x_{n-1}x_n^{d-1} \quad x_n^d ] .
\ee
For a symmetric matrix $A$, 
the 
$A\succeq 0$ (resp., $A\succ 0$)
means that $A$ is positive semidefinite (resp., positive definite).
A property is said to hold \emph{generically} if it holds
in an open dense subset of the embedding vector space
for input data (e.g., the coefficient vectors of defining polynomials of GNEPs)
in the Zariski topology.
We refer to \cite{nie2022algebraic} for more details on genericity properties of GNEPs.

\subsection{Ideals and quadratic modules}
For a polynomial $p \in \mathbb{R}[x]$ and subsets $I,J \subseteq \mathbb{R}[x]$, define the product and the Minkowski sum
\begin{equation*}
p \cdot I \coloneqq \{pq : q \in I\}, \,\,\, I+J \coloneqq \{a+b : a \in I, b \in J\}.
\end{equation*}
A subset $I \subseteq \mathbb{R}[x]$ is an ideal of $\mathbb{R}[x]$
if $I+I \subseteq I$ and $p \cdot I \subseteq I$ for all $p \in \re[x]$.
For a tuple $h = (h_1, \ldots, h_m)$ of polynomials in $\mathbb{R}[x]$,
the ideal generated by $h$ is
\begin{equation*}
\mbox{Ideal}[h]\,  \coloneqq \, h_1 \cdot \mathbb{R}[x] + \cdots + h_m \cdot \mathbb{R}[x] .
\end{equation*}
For a degree $d\ge \deg(h)$, we denote the degree-$d$ truncation of $\mbox{Ideal}[h]$
as
\[
\mbox{Ideal}[h]_d \coloneqq h_1\cdot \re[x]_{d-\deg(h_1)}+\cdots+h_m\cdot \re[x]_{d-\deg(h_m)}.
\]
The real zero set of $h$ is
$Z(h) \coloneqq \{x \in \mathbb{R}^n: h_1(x) = \cdots = h_m(x) = 0\}.$

A polynomial $\sigma \in \mathbb{R}[x]$ is said to be a sum of squares (SOS)
if there are polynomials $p_1, \ldots, p_k \in \mathbb{R}[x]$ such that $\sigma = p_1^2 + \cdots + p_k^2$.
The set of all SOS polynomials in $x$ is denoted as $\Sigma[x]$.
In computation, we often work with the degree-$d$ truncation
$\Sigma[x]_d \coloneqq \Sigma[x] \cap \mathbb{R}[x]_d.$
For a polynomial tuple $g \coloneqq (g_1, \ldots, g_t)$, denote
$S(g) \coloneqq \{x \in \mathbb{R}^n : g_1(x) \geq 0, \ldots, g_t(x) \geq 0\}.$
The quadratic module generated by $g$ is:
\begin{equation*}
\mbox{QM}[g] \coloneqq \Sigma[x] + g_1 \cdot \Sigma[x] + \cdots + g_t \cdot \Sigma[x].
\end{equation*}
The degree-$d$ truncation of $\mbox{QM}[g]$ is similarly defined as
\begin{equation*}
\mbox{QM}[g]_d \coloneqq \Sigma[x]_d + g_1 \cdot \Sigma[x]_{d-\deg(g_1)} +  \cdots + g_t \cdot \Sigma[x]_{d-\deg(g_t)}.
\end{equation*}

If $f \in \mbox{Ideal}[h]+\mbox{QM}[g]$, then it is clear that $f \geq 0$ on $Z(h) \cap S(g)$. 
The reverse is not necessarily true.
The set $\mbox{Ideal}[h]+\mbox{QM}[g]$ is said to be \textit{archimedean}
if there exists $q \in \mbox{Ideal}[h]+\mbox{QM}[g]$ such that
$S(q)$ is a compact set. When
$\mbox{Ideal}[h]+\mbox{QM}[g]$ is archimedean, if $f>0$ on $Z(h) \cap S(g)$,
then $f \in  \mbox{Ideal}[h]+\mbox{QM}[g]$.
This conclusion is referenced as \textit{Putinar's Positivstellensatz} \cite{Putinar}.
Interestingly, if $f \geq 0$ on $Z(h) \cap S(g)$, we also have $f \in\mbox{Ideal}[h]+\mbox{QM}[g]$, under some standard optimality conditions \cite{nie2014optimality}.

\subsection{Localizing and moment matrices}
A real vector $y$ is called a \textit{truncated multi-sequence} (tms) of degree $2k$ if it is labeled as
$y \,=\, (y_\alpha)_{\alpha \in \mathbb{N}^n_{2k}}.$
For a tms $y\in\re^{\mathbb{N}^n_{2k}}$ and a polynomial
$f = \sum_{\alpha \in \mathbb{N}^n_{2k}} f_{\alpha}x^\alpha$,
define the operation
\be\label{def<f,y>}
\langle f, y \rangle \,\coloneqq\, \sum\limits_{\alpha \in \mathbb{N}^n_{2k}} f_{\alpha}y_\alpha.
\ee
In particular, if $y = [\hat{x}]_{2k}$ for some $\hat{x}\in\re^n$,
then $\langle f,y\rangle = \langle f, [\hat{x}]_{2k}\rangle = f(\hat{x})$,
where $[\hat{x}]_{2k}$ is the monomial vector as in \eqref{[x]d}.
For $q \in \mathbb{R}[x]_{2k}$ and $t  = k - \lceil\mbox{deg}(q)/2 \rceil$,
the product $q \cdot [x]_t[x]_t^T$ is a symmetric matrix polynomial of length $\binom{n+t}{t}$,
which can be expressed as
$q \cdot [x]_t[x]_t^T = \sum_{\alpha \in \N^n_{2k}} x^\alpha Q_\alpha$
for some symmetric matrices $Q_\alpha$.
For $y \in\re^{\N^n_{2k}}$, denote the matrix
$L_q^{(k)}[y] \coloneqq \sum_{\alpha \in \mathbb{N}^n_{2k}} y_\alpha Q_\alpha.$
It is called the $k$th order \textit{localizing matrix} of $q$ generated by $y$.
In particular, if $q=1$ (the constant $1$ polynomial), then
$L_q^{(k)}[y]$ is reduced to the moment matrix
$M_k[y] \,\coloneqq\, L_1^{(k)}[y].$

\subsection{Moment-SOS relaxations}
Quadratic modules, moment, and localizing matrices are useful
for solving polynomial optimization.
Consider the polynomial optimization problem:
\begin{equation}\label{eq:PolyOpt}
\left\{\begin{array}{cl}
    \min\limits_{x\in\re^n} & f(x)\\
    \st & h(x) = 0,\quad g(x)\ge 0,
\end{array}
\right.
\end{equation}
For each positive integer $k$ such that $2k\ge \max \{\deg(f),\deg(h),\deg(g) \}$, the $k$th order SOS relaxation of the above problem is:
\begin{equation}\label{eq:SOS_k}
\left\{\begin{array}{cl}
\max\limits_{\gamma\in\re} & \gamma\\
\st & f-\gamma\in \mbox{Ideal}[h]_{2k}+\qmod{g}_{2k}.
\end{array}
\right.
\end{equation}
When $\mbox{Ideal}[h]+\qmod{g}$ is archimedean, the optimal value of \eqref{eq:SOS_k} converges to that of \eqref{eq:PolyOpt}
as $k$ goes to infinity.
Problem \eqref{eq:SOS_k} is a semidefinite program,
whose dual problem is the $k$th order moment relaxation of \eqref{eq:PolyOpt}:
\begin{equation}\label{eq:mom_k}
\left\{\begin{array}{cl}
\min & \langle f, y\rangle\\
\st & y\in\re^{\N_{2k}^n},\, y_0 = 1,\,M_k[y]\succeq 0,\\
& L_{h_i}^{(k)}[y] = 0\,(i\in[m]),\\
 & L_{g_j}^{(k)}[y]\succeq 0\,(j\in[t]).
\end{array}
\right.
\end{equation}
The primal-dual pair \eqref{eq:SOS_k}--\eqref{eq:mom_k} forms
the $k$th order level of the Moment-SOS hierarchy associated with \eqref{eq:PolyOpt}.
Polynomial optimization and the Moment-SOS hierarchy are also discussed in Section~\ref{sc:pops}
with more details.
We refer to \cite{HenrionKordaLasserre2020,lasserre2015,Lau09,MPO}
for a thorough introduction to this field.

\subsection{Carath\'{e}odory's theorem and Lagrange multipliers}
Carath\'{e}odory's theorem is useful for deriving pLMEs for GNEPs.

\begin{theorem}{\cite[Theorem~1.2.4]{MPO}}\label{tm:cath}
Let $S$ be a nonempty set of a vector space $V$ of dimension $m$.
Then, every point in the conic hull of $S$ can be expressed as a conic combination
of at most $m$ points.
\end{theorem}

The following is an example of applying
Carath\'{e}odory's theorem to get pLMEs.

\begin{example}
Consider the linear programming:
\be \label{eq:LP}
\left\{\begin{array}{cl}
\min\limits_{x \in \re^{3}} & x_1 + x_2 + x_3 \\
\st & Ax\ge b,
\end{array}
\right.
\ee
where $b =
 \begin{bmatrix}
0 & 0 & 0 & 1 & 1 & 1 & 2 & 3 & 3 & 2
\end{bmatrix}^T \in \re^{10}$ and
\[
A = \begin{bmatrix}
1 & 0 & 0 & 1 & 1 & 0 & 1 & 2 & 1 & 2 \\
0 & 1 & 0 & 1 & 0 & 1 & 1 & 1 & 2 & 1 \\
0 & 0 & 1 & 0 & 1 & 1 & 2 & 0 & 0 & 1
\end{bmatrix}^T\in \re^{10\times 3}.
\]
The set of minimizers of \eqref{eq:LP} is identical to the set of KKT points.
Let $u = (1,1,0)^T.$
We verify that $u$ is the minimizer of \eqref{eq:LP} by checking the KKT conditions.
Note that
\[ Au-b = \bbm 1\ \ 1\ \ 0\ \ 1\ \ 0\ \ 0\ \ 0\ \ 0\ \ 0\ \ 1\ebm^T.\]
Since $u$ is active at six constraints labeled with
$J_1 = \{3,5,6,7,8,9\}$, if it is a KKT point, then the corresponding Lagrange multiplier vector
$\lmd = (\lmd_1\ddd \lmd_{10})$ must satisfy $\lambda_j = 0$ for all $j\notin J_1$, and
\[
A^T\lmd = \mathbf{1},\quad \lambda_j\ge 0,\quad\forall j\in J_1.
\]
The above is an underdetermined system over the nonnegative orthant,
which has multiple solutions.
By Carath\'{e}odory's theorem, it has a solution with at most three nonzero entries.
Let $J_2 = \{3,8,9\}\subseteq J_1$ and consider the following induced system
\[
\bbm 0 & 2 & 1\\0 & 1 & 2\\1 & 0 & 0\ebm\bbm \lambda_3\\ \lambda_8\\ \lambda_9\ebm
= \bbm 1\\1\\1\ebm,\quad \lambda_j\ge 0\,(\forall j\in J_2),\quad
\lambda_j = 0\,(\forall j\notin J_2).
\]
It admits a unique solution that gives the vector of Lagrange multipliers such that
\[\lmd_3 = 1, \quad \lmd_8=\lmd_9=\frac{1}{3}, \quad \lmd_j=0\,(\forall j\notin J_2). \]
This implies that $u$ is an optimizer of \eqref{eq:LP}.
\end{example}

\section{Partial Lagrange Multiplier Expressions}
\label{sec:KKTpLME}

This section discusses how to find a convenient representation for the KKT set with {\it partial Lagrange multiplier expressions} (pLMEs).
We consider GNEPs with quasi-linear constraints.
The $i$th player's decision optimization problem $\mbox{F}_i(x_{-i})$ reads
\be \label{GNEP:Axi>=bi}
\left\{\begin{array}{cl}
\min\limits_{x_i\in\re^{n_i}} & f_i(x_i, x_{-i})\\
\st & A_ix_i-b_i(x_{-i})\ge 0,
\end{array}
\right.
\ee
where $A_i \in  \re^{m_i\times n_i}$ and $b_i(x_{-i})$ is a polynomial vector in $x_{-i}$.
Recall the notation
\[
A_i \, = \, \bbm a_{i,1} & \cdots &  a_{i,m_i} \ebm^T.
\]
Since the constraints of \reff{GNEP:Axi>=bi} are linear,
for every optimizer $x_i\in \mc{S}_i(x_{-i})$, there exists a Lagrange multiplier vector
$\lambda_i = (\lambda_{i,1},\ldots, \lambda_{i,m_i})$ such that
\be \label{eq:KKT:sp}
\boxed{
\begin{array}{c}
\nabla_{x_i} f_i(x) - A_i^T \lambda_i = 0, \\
0  \le  \lambda_i  \perp ( A_ix_i-b_i(x_{-i}))  \ge  0 .
\end{array}
}
\ee
The set of all KKT points is
\be \label{eq:mcK}
\mc{K} \,=\, \left\{ x\in X\left|
\begin{array}{c}
\mbox{$\exists \, (\lambda_1,\ldots, \lambda_N)$ such that for each $i\in[N]$,}  \\
\nabla_{x_i} f_i(x) - A_i^T\lambda_i = 0 \,\, ,  \\
0\le \lambda_{i}\perp ( A_ix_i-b_i(x_{-i}) ) \ge 0
\end{array}\right.\right\}.
\ee
When $\rank\, A_i = m_i \le n_i$, we can get the following expression for $\lmd_i$:
\[\lambda_i \,=\, (A_i A_i^T)^{-1}A_i\nabla_{x_i} f_i(x).\]
When $m_i>n_i$, the above expression is not available 
since $A_i A_i^T$ is singular.
Indeed, for the case $m_i>n_i$,  a polynomial expression for $\lmd_i$
typically does not exist, since $b_i(\xmi)$ depends on $\xmi$,
but a rational expression for $\lmd_i$ always exists.
This is shown in \cite{nie2023convex,NieTangZgnep21}.
However, such an expression for $\lmd_i$ may be too complicated to be practical.
The expression becomes more complicated if the gap $m_i-n_i$ is large
(see \cite[Proposition~4.1]{Nie2019}).

We look for more convenient expressions for $\lmd_i$. Let
\be \label{rank:ri}
r_i \,\coloneqq \, \mbox{rank}\, A_i .
\ee
For each $x \in \mc{K}$, the KKT system \reff{eq:KKT:sp}
has a solution $\lmd_i$ that has at most $r_i$ nonzero entries by Carath\'{e}odory's Theorem.
This means that $m_i-r_i$ entries of such a $\lmd_i$ must be zeros.
If the label set $J_i$ of nonzero entries of $\lmd_i$ is known, then the expression for $\lmd_i$ can be simplified.
This gives a pLME for $\lmd_i$.

\subsection{The pLMEs}\label{sec:pLME}

To find pLMEs, consider the linear system
\be   \label{eq:df=Almd}
\nabla_{x_i}f_i(x) \,=\, A_i^T\lambda_i, \quad
\lambda_{i} \ge 0.
\ee
For a subset $J_i\subseteq [m_i]$,
let $A_{i,J_i}$ denote the submatrix of $A_i$
whose row labels are in $J_i$, and so is $\lmd_{i,J_i}$.
That is,
\[
A_{i,J_i}  \coloneqq  \bbm a_{i,j}^T \ebm_{j\in J_i} , \quad
\lambda_{i,J_i} \coloneqq \bbm \lambda_{i,j} \ebm_{j\in J_i}.
\]
Let $r_i$ be the rank as in \reff{rank:ri}.
Denote the set of label sets
\be   \label{eq:Pmr}
\mc{P}_i \,\coloneqq\, \Big \{J_i\subseteq [m_i]:
 |J_i| = r_i, \,\, \mbox{rank}\,A_{i,J_i} = r_i \Big \}.
\ee
Since $A_i$ is $m_i$-by-$n_i$, we know $r_i\le \min\{m_i,n_i\}$
and each $\mc{P}_i$ is nonempty.
For each $J_i\in\mc{P}_i$, the vector
$\lambda_i = (\lambda_{i,1},\ldots, \lambda_{i,m_i})\ge 0$
is said to be a {\it basic feasible solution} of \reff{eq:df=Almd}
with respect to $J_i$ if $\lambda_{i,j} = 0$ for all $j \not \in J_i$ and
\begin{equation}\label{eq:kktbs}
\nabla_{x_i}f_i(x) \,=\, \sum\limits_{j\in J_i}\lambda_{i,j} a_{i,j}
\,=\,  A_{i,J_i} ^T\lambda_{i,J_i}.
\end{equation}
Basic feasible solutions can be conveniently expressed by pLMEs.
Multiplying $A_{i,J_i}$ on both sides of \reff{eq:kktbs}, we get
\[
(A_{i,J_i}A_{i,J_i}^T)\lambda_{i,J_i} \,=\, A_{i,J_i}\nabla_{x_i} f_i(x).
\]
Since $\rank \, A_{i,J_i} = r_i$, this gives the pLME:
\be  \label{eq:pLME}
\lmd_{i,J_i} \,=\, \lambda_{i,J_i}(x) \,\coloneqq\,
(A_{i,J_i} A_{i,J_i}^T)^{-1}A_{i,J_i}\nabla_{\xpi} f_i(x).
\ee
In particular, for the case $r_i = n_i$, the above simplifies to
\be  \label{pLME:ri=ni}
\lambda_{i,J_i}(x)  \,=\,  (A_{i,J_i})^{-T}\nabla_{x_i}f_i(x).
\ee
Here, the superscript denotes the transpose of the inverse.

\begin{example}\label{ex:pLME:iri}
Consider the 2-player GNEP with
\begin{gather}
\text{F}_1(x_{-1}) :\left\{\begin{array}{cl}
\min\limits_{x_1\in\re^2} & \|x_1\|^2\\
\st & \left[\begin{array}{rr} -1 & -1 \\1 & 0\\ 0 & 1\\6 & 1
\end{array}\right]x_1 \geq \left[\begin{array}{c} -2 \\ 0 \\ 0 \\ 1 + \mathbf{1}^Tx_2 \end{array}\right],
\end{array}
\right. \\
\text{F}_2(x_{-2}) :\left\{\begin{array}{cl}
\min\limits_{x_2\in\re^2} & \|x_2\|^2\\
\st & \left[\begin{array}{rr}
-1 & 1\\1 & 0\\0 & 1\\-1 & 0\end{array}\right]x_2 \geq \left[\begin{array}{c} -2x_{1,1}+2 \\ 0 \\ 0 \\ x_{1,1}-x_{1,2}-2 \end{array}\right].
\end{array}
\right.\end{gather}
One can see that $m_1 = m_2 = 4$, $r_1 = r_2 = 2$ and
\[
\begin{array}{lll}
\mc{P}_1  = \big \{\, \{1, 2\},\, \{1,3\},\, \{1, 4\},\, \{2, 3\}, & \vline &
\mc{P}_2  = \big \{\, \{1, 2\},\, \{1,3\},\, \{1, 4\},\, \{2, 3\},\\
\qquad\quad \{2, 4\},\, \{3, 4\}\, \big \}, &\vline &
\qquad\quad \{3, 4\} \, \big \}.
\end{array}
\]
The pLMEs are given as in \reff{pLME:ri=ni}. For instance,
\begin{equation*}
\lambda_{1, \{1,4\}} = \left[\begin{array}{rr}
-1 & -1 \\ 6 & 1
\end{array}\right]^{-T} \left[\begin{array}{c}
2x_{1,1} \\ 2x_{1,2}
\end{array}\right] = \frac{2}{5}\left[\begin{array}{c}
x_{1,1} -6x_{1,2} \\ x_{1,1} - x_{1,2}
\end{array}\right],
\end{equation*}
\begin{equation*}
\lambda_{2, \{1,2\}} = \left[\begin{array}{rr}
-1 & 1 \\ 1 & 0
\end{array}\right]^{-T} \left[\begin{array}{c}
2x_{2,1} \\ 2x_{2,2}
\end{array}\right] = 2\left[\begin{array}{c}
x_{2,2} \\ x_{2,1} + x_{2,2}
\end{array}\right].
\end{equation*}
In contrast, if we do not use pLMEs,
the full Lagrange multiplier expressions as in
\cite{nie2023convex,NieTangZgnep21}
are much more complicated for this GNEP.
\end{example}

\begin{remark}
For the GNEP (\ref{eq:GNEP}) with quasi-linear constraints, the pLMEs always exist for each $i\in [N]$.
Furthermore, at every KKT point $x$, there always exist an associated Lagrange multiplier vector $\lmd_i\in\re^{m_i}$ and a label set $J_i\in \mc{P}_i$ such that $\lmd_{i,J_i}(x) = \lmd_{i,J_i}$ and $\lmd_{i,j} = 0\ (j\notin J_i)$ for all $i\in [N]$.
In other words, the pLMEs give exact values for Lagrange multipliers at every KKT point of the GNEP.
\end{remark}

For $x = (x_i,x_{-i})\in X$ and $J_i\in\mc{P}_i$,
if $\lambda_{i,J_i}(x)\ge 0$, then $x$ is a KKT point of $\mbox{F}_i(x_{-i})$,
since \reff{eq:kktbs} is satisfied for $\lambda_{i,J_i}(x)$.
Thus, $x_i$ is a KKT point of $\mbox{F}_i(x_{-i})$ if
\begin{equation}\label{eq:kidcp}
\boxed{
\begin{array}{c}
\exists J_i\in \mc{P}_i , \,
\nabla_{x_i} f_i(x) = A_{i,J_i}^T\lambda_{i,J_i}(x),\\
0  \le \lambda_{i,J_i}(x)  \perp \big( A_{i,J_i}x_i-b_{i,J_i}(x_{-i}) \big)  \ge 0 .
\end{array}
}
\end{equation}
Interestingly, the above is also necessary for
$x_i$ to be a KKT point of $\mbox{F}_i(x_{-i})$,
as shown in the following theorem.

\begin{theorem}\label{thm:Kidecomp}
For $x = (x_i,x_{-i})\in X$, the $x_i$ is a KKT point of $\mbox{F}_i(x_{-i})$
if and only if it satisfies \reff{eq:kidcp}.
\end{theorem}
\begin{proof}
Suppose \reff{eq:kidcp} is satisfied.
Let $\lambda_i$ be the extension of $\lambda_{i,J_i}(x)$
by adding zero entries. Then $(x,\lambda_i)$ satisfies the KKT system \reff{eq:KKT:sp},
so $x_i$ is a KKT point of $\mbox{F}_i(x_{-i})$.
Conversely, suppose $x_i$ is a KKT point of $\mbox{F}_i(x_{-i})$.
Then there exists $\lmd_i$ satisfying \reff{eq:KKT:sp}.
So, the solution set for the linear system \reff{eq:df=Almd} is nonempty.
By Carath\'{e}odory's Theorem, $\nabla_{x_i}f_i(x)$ can be represented as a conic combination of linearly independent vectors from $A_i^T$. Thus, a
basic feasible solution must exist for \reff{eq:df=Almd}.
This means that \reff{eq:kidcp} holds.
\end{proof}

\subsection{Expression of the KKT set}

For the label sets $\mc{P}_1,\ldots,\mc{P}_N$ as in \reff{eq:Pmr}, define the Cartesian product
\be   \label{eq:setP}
\mc{P} \,\coloneqq\, \mc{P}_1 \times \cdots \times \mc{P}_N.
\ee
Table~\ref{table:numberP} shows some typical instances of $|\mathcal{P}|$
when  each $A_i$ is a generic matrix such that $\rank\, A_i = n_i$ and every $b_i$ is a generic polynomial in $\xmi$.
In the table, $|\mc{A}|$ denotes the number
of all possibilities of active constraints.
One can see that $|\mc{P}| \ll |\mc{A}|$.
Moreover, when $A_i$ are not generic, the actual values of $|\mathcal{P}|$ are smaller than or equal to the values in Table~\ref{table:numberP}, while the values of $|\mathcal{A}|$ may be greater than those shown in the table due to the degeneracy of constraints.
\begin{table}[htb]
\centering
\caption{Some examples of $|\mc{P}|$ and $|\mc{A}|$ when $\rank\, A_i = n_i$ for all $i$.}
\label{table:numberP}
\begin{tabular}{|c|c|c||c|c|c|}
	\hline
 $\begin{array}{c}
   (n_1, \ldots, n_N) \\
   (m_1, \ldots, m_N)
\end{array}$ & $|\mc{P}|$ & $|\mc{A}|$  & $ \begin{array}{c}
   (n_1, \ldots, n_N)\\
   (m_1, \ldots, m_N)
\end{array}$ & $ |\mc{P}|$ & $ |\mc{A}|$ \\ \hline
$\begin{array}{c}
     (2,2)\\(5,5)
\end{array}$ & $ 100$ & $ 256 $ & $\begin{array}{c}
     (2,2,2)\\(5,5,5)
\end{array}$ & $ 1000$ & $ 4096$ \\ \hline
$\begin{array}{c}
     (2,4)\\(5,7)
\end{array}$ & $ 350$ & $ 1584 $ & $\begin{array}{c}
     (1,2,3)\\(2,3,4)
\end{array}$ & $24$ & $ 315 $ \\ \hline
$\begin{array}{c}
    (2,4)\\(4,8)
\end{array}$ & $ 420$ & $ 1793 $ & $ \begin{array}{c}
     (3,3,3)\\(6,6,6)
\end{array}$ & $8000$ & $ 74088 $ \\ \hline
$\begin{array}{c}
     (3,3) \\ (5,5)
\end{array}$ & $ 100$ & $ 676$ & $  \begin{array}{c}
    (2,2,2,2)\\(4,4,4,4)
\end{array}$ & $1296$ & $ 14641$ \\ \hline
$\begin{array}{c}
     (4,4) \\ (7,7)
\end{array}$ & $ 1225$ & $ 9801 $ & $ \begin{array}{c}
(3,3,3,3)\\(5,5,5,5)
\end{array}$ & $10000$ & $456976$ \\ \hline
\end{tabular}
\end{table}

For a tuple $J = (J_1,\ldots, J_N) \in  \mc{P}$ with each $J_i\in\mc{P}_i$,
let $\lambda_{i,J_i}(x)$ be the pLME given by \reff{eq:pLME}
and define the set
\be\label{eq:KJ} \mc{K}_J   \, \coloneqq  \,  \left\{ x\in X  \, \left|\baray{c}
  \nabla_{\xpi} f_i(x) -  A_{i,J_i}^T \lambda_{i,J_i}(x)=0,\\
  0 \le \lambda_{i,J_i}(x)\perp \big( A_{i,J_i} x_i - b_{i,J_i}( x_{-i} )  \big) \geq 0 , \\
  \text{for all} \quad   i = 1, \ldots, N
\earay\right. \right\}.
\ee
Clearly, each $x \in \mc{K}_J$ is a KKT point for the GNEP \reff{eq:GNEP},
so $\mc{K}_J\subseteq \mc{K}$. Indeed, every KKT point
belongs to $\mc{K}_J$ for some $J$.
This is shown in the following theorem.

\begin{theorem}\label{thm:k=ukj}
For the GNEP given by \reff{GNEP:Axi>=bi},
the KKT set $\mc{K}$ can be expressed as
\be  \label{eq:Kdcp}
 \mc{K} \, = \, \bigcup_{J \in \mc{P}}  \mc{K}_J .
\ee
\end{theorem}
\begin{proof}
By Theorem~\ref{thm:Kidecomp},
the KKT set for the optimization $\mbox{F}_i(x_{-i})$ is
\[
\widehat{\mc{K}}_i(x_{-i}) \, \coloneqq   \bigcup_{  J_i\in \mc{P}_i }
\left\{ x_i  \left|
\baray{c}
   x_i \in X_i(x_{-i}), \quad
  \nabla_{\xpi} f_i(x) -  A_{i,J_i}^T \lambda_{i,J_i}(x)=0,\\
 0 \le \lambda_{i,J_i}(x)\perp \big( A_{i,J_i} x_i - b_{i,J_i}( x_{-i} )  \big)  \geq 0
\earay
\right. \right\} .
\]
In view of \reff{eq:KJ}, we have
\[
\mc{K} \,=\, \bigcap_{i\in [N]} \big \{x \in X: x_i \in \widehat{\mc{K}}_i(x_{-i}) \big \}
\,=\, \bigcup_{J\in \mc{P}}\mc{K}_J.
\]
So, equation \reff{eq:Kdcp} holds.
\end{proof}

When each $\rank \, A_i = n_i$, pLMEs can be given as in \reff{pLME:ri=ni}
and Theorem~\ref{thm:k=ukj} implies the following
simplified expression.

\begin{corollary}  \label{prop:KJ}
If $\rank\, A_i = n_i$ for each $i$, then
\begin{equation*}
 \mc{K}_J  =  \left\{ x \in X\,
 \left|
  \baray{c}
 0 \le A_{i,J_i}^{-T}\nabla_{x_i}f_i(x)    \perp
  \big( A_{i,J_i} x_i - b_{i,J_i}( x_{-i} )  \big)  \geq 0,   \\
  \text{for all} \quad i = 1, \ldots, N
\earay \right.
\right\}
\end{equation*}
for every $J = (J_1,\ldots, J_N)\in \mc{P}$ and
\be  \label{K:ri=ni}
 \mc{K} \,=\, \bigcup_{J \in \mc{P}}
     \left\{ x \in X\,
 \left|   \baray{c}
 0 \le A_{i,J_i}^{-T}\nabla_{x_i}f_i(x)    \perp
  \big( A_{i,J_i} x_i - b_{i,J_i}( x_{-i} )  \big)  \geq 0, \  \\
  \text{for all} \quad  i =1, \ldots, N
\earay \right.
\right\}.
\ee
\end{corollary}

\begin{example}
For the GNEP in Example~\ref{ex:pLME:iri},
it is clear that $r_i = n_i = 2$ for $i = 1,2$.
For $J = (J_1,J_2)$ with $J_1 = \{1,4\}$ and $J_2 = \{1,2\}$,
the $\mc{K}_J$ is given by
\[
\begin{array}{ll}
2-\mathbf{1}^Tx_1 \ge 0,\quad x_1\ge 0,&   2x_{1,1}-x_{2,1}+x_{2,2}-2 \ge 0,\\
6x_{1,1}+x_{1,2}-\mathbf{1}^Tx_2-1 \ge 0, &  x_2\ge 0,\quad 2-x_{1,1}+x_{1,2}-x_{2,1}\ge 0,\\
x_{1,1}-6x_{1,2} \ge0,\, x_{1,1}-x_{1,2}\ge 0, &  x_{2,1}+x_{2,2}\ge 0,\\
(x_{1,1}-6x_{1,2})(2-\mathbf{1}^Tx_1) = 0, & x_{2,2}(2x_{1,1}-x_{2,1}+x_{2,2}-2) = 0,\\
(x_{1,1}-x_{1,2})(6x_{1,1}+x_{1,2}-\mathbf{1}^Tx_2-1) = 0, &
(x_{2,1}+x_{2,2})x_{2,1}= 0.
\end{array}
\]
Indeed, one can further verify that $\mc{K}_J= \big \{  (18, 3, 0, 62 )/49 \big \}$, a singleton.
Furthermore, this point is the unique GNE as well.
By Algorithm \ref{alg:pLME_all}, we know that it is contained in
$\mc{K}_J$ not only for $J = (\{1,4\},\{1,2\})$ but also for
$J = (\{2,4\},\{1,2\})$ or $J = (\{3,4\},\{1,2\})$.
For the GNE, the active label set is
$\hat{J} = (\hat{J}_1,\hat{J}_2)$,
with $\hat{J}_1 = \{4\}, \hat{J}_2 = \{1,2\}$,
and $\hat{J}_1 \subsetneqq  J_1$.
That is, the label set $J$ for finding this KKT point
is not the active constraining set $\hat{J}$.
We remark that the expressions in \reff{eq:Kdcp} and \reff{K:ri=ni}
are not just enumerations of active constraining sets.
\end{example}

\section{Solving GNEPs with quasi-linear constraints}
\label{sec:findGNEs}

We discuss how to solve the GNEP with quasi-linear constraints as in \reff{GNEP:Axi>=bi}.
Since every GNE $x$ is a KKT point, there exists
$J\in\mc{P}$ such that $x\in\mc{K}_J$.
Since $\mc{P}$ is a finite set, there are only finitely many choices of $J$.
Moreover, under some genericity assumptions, the KKT set $\mc{K}$ is finite.
For these cases, the subset $\mc{K}_J$ is also finite for every $J\in\mc{P}$.
This inspires how to find all GNEs.

\subsection{Finding all GNEs in {$\mc{K}_J$}{KJ}}

We introduce how to find GNEs in $\mc{K}_J$ for a fixed $J\in\mc{P}$.
For the given $J$, pLMEs are given by (\ref{eq:pLME}), so the set $\mc{K}_J$ can be represented by equalities and inequalities of polynomials in the variable $x$, as shown in (\ref{eq:KJ}).
Let $\Theta \in\re^{(n+1)\times (n+1)}$ be a symmetric positive definite matrix.
Consider the following polynomial optimization problem
\be \label{eq:findfirstKKT}
\, \left\{
\begin{array}{cl}
\min &  \theta(x)  \coloneqq  [x]_1^T\Theta[x]_1\\
\st & x\in\mc{K}_J.
\end{array}
\right.
\ee
If $\mathcal{K}_J \neq \emptyset$, then \reff{eq:findfirstKKT}
has a unique minimizer $u$ when $\Theta$ is generic
(see \cite[Theorem~5.4]{nie2023convex}), and $u$ is a KKT point.
Otherwise, (\ref{eq:findfirstKKT}) is infeasible, and there is no GNE in $\mc{K}_J$.
We will show how to solve (\ref{eq:findfirstKKT}) in Section~\ref{sc:pops}.
The following result is obvious.

\begin{theorem}\label{eq:solveone}
For the GNEP as in \reff{GNEP:Axi>=bi},
if the optimization problem (\ref{eq:findfirstKKT}) is infeasible, then there is no KKT point in $\mc{K}_J$.
Otherwise, each minimizer $u$ of (\ref{eq:findfirstKKT}) is a KKT point.
Moreover, if the GNEP is convex, $u$ is a GNE.
\end{theorem}

For convex GNEPs, once we find a minimizer $u$ for (\ref{eq:findfirstKKT}),
then $u$ must be a GNE.
However, when the GNEP is nonconvex, $u$ may or may not be a GNE.
For nonconvex GNEPs, we can check if $u$ is a GNE or not by solving polynomial optimization problems.
By definition, $u$ is a GNE if and only if $\epsilon_i\ge 0$ for every $i\in[N]$,
where $\eps_i$ is the optimal value
\begin{equation}
\label{eq:verifyGNE}
\left\{ \begin{array}{ccl}
\eps_i \coloneqq &\min  &  f_i(x_i, u_{-i}) - f_i(u_{i}, u_{-i}) \\
&\st & x_i \in X_i(u_{-i}).
\end{array} \right.
\end{equation}
Therefore, once we get a KKT point $u$, we solve (\ref{eq:verifyGNE})
for every $i\in[N]$.
If $\epsilon_i\ge0$ for all $i\in[N]$, then we certify that $u$ is a GNE;
otherwise, it is not.

If $u$ is not a GNE, one needs to find other KKT points to solve this GNEP.
Also, when $\mc{K}_J\ne \emptyset$ but $\mc{K}_J\cap \mc{S} = \emptyset$,
we may need to find all points in $\mc{K}_J$ to certify nonexistence of GNEs in $\mc{K}_J$.
Moreover, people are usually interested in finding all GNEs.
In the following, we discuss how to find all GNEs or detect their nonexistence
in $\mc{K}_J$.

Suppose $u = (u_1, \ldots, u_N) \in \mc{K}_J$
is the minimizer of (\ref{eq:findfirstKKT}).
Then we have
\[
\theta( u ) \,  \le  \,  \theta( x ) \quad
\text{for all} \,\,  x\in\mc{K}_J.
\]
When the matrix $\Theta$ is generic, the above inequality holds strictly for all $x\in\mc{K}_J\setminus \{u\}$.
Suppose that $u$ is an isolated point of $\mc{K}_J$.
This is the case when the GNEP is generic, as shown in \cite[Theorem~3.1]{nie2022algebraic}.
Then, there exists $\delta>0$ such that
\be\label{eq:correct_delta}
\theta( u ) + \dt \le \theta( x ) ,
\quad \text{for all} \,\, u \ne   x  \in  \mc{K}_J .
\ee
For such a $\delta>0$, consider the optimization problem
\be  \label{eq:branch_delta}
\, \left\{
\begin{array}{cl}
\min &   \theta( x )  \\
\st & x\in\mc{K}_J,\quad  \theta( x )  \ge  \theta( u )  + \dt.
\end{array}
\right.
\ee
If \reff{eq:branch_delta} is infeasible,
then the set $\mc{K}_J$ does not have any KKT points other than $u$.
Otherwise, it must have a minimizer $\hat{u}$ (since $\Theta$ is positive definite),
which is a KKT point different from $u$.
For the new KKT point $\hat{u}$, we may solve polynomial optimization problems like (\ref{eq:verifyGNE}) to check if it is a GNE or not.

Indeed, more GNEs can be computed by repeating these steps.
Suppose we have obtained the KKT points
$u^{(1)}, u^{(2)}, \ldots,  u^{(j)} \in \mc{K}_J$ for some $j \ge 1$, in the order that
\be \label{uj:order}
\theta( u^{(1)} )  < \theta( u^{(2)} )  < \cdots <  \theta( u^{(j)} ) .
\ee
Suppose $u^{(j+1)}$ is a new KKT point such that
\[
\theta( u^{(j+1)} )  \, = \, \min_{x \in \mc{C}_j}  \theta( x),
\quad \text{where} \quad \mc{C}_j \coloneq \mc{K}_J \setminus
\{ u^{(1)}, u^{(2)}, \ldots,  u^{(j)} \}.
\]
If there exists a scalar $\dt$ satisfying
\be  \label{range:dt}
0 <  \delta < \theta( u^{(j+1)} ) -  \theta( u^{(j)} ),
\ee
then $u^{(j+1)}$ can be obtained by computing the minimizer of
\be  \label{theta>=uj}
\, \left\{
\begin{array}{cl}
\min &   \theta( x )  \\
\st  &   \theta( x )  \ge  \theta( u^{(j)} )  + \dt,  \\
     &    x\in\mc{K}_J .
\end{array}
\right.
\ee
The inequality \reff{range:dt} can be checked as follows.
We can first assign a priori value for $\delta$ (say, $0.5$),
then solve the maximization problem
\be  \label{eq:branch_max}
\left\{\begin{array}{rcl}
\theta_{\max} \coloneqq  & \max &   \theta(x)  \\
& \st & x\in\mc{K}_J,\quad \theta(x) \le \theta(u^{(j)}) + \dt.
\end{array}
\right.
\ee
Since $u^{(j)}$ is a feasible point,
it always holds $\theta_{\max} \ge \theta(u^{(j)})$.
There are two possibilities:

\bit

\item
If $\theta_{\max} = \theta(u^{(j)})$, then $u^{(j)}$
is a maximizer of (\ref{eq:branch_max}).
This implies that $u^{(j+1)}$ is infeasible for \reff{eq:branch_max},
so \reff{range:dt} is satisfied.

\item
If $\theta_{\max} > \theta(u^{(j)})$, then there exists $v \in \mc{K}_J$ such that
\[
\theta(u^{(j)}) <   \theta(v) \le   \theta(u^{(j)}) + \dt .
\]
This means that $\delta$ is too large and violates \reff{range:dt}.
We need to decrease the value of $\delta$ (e.g., by replacing $\delta$ with $\delta/2$)
and solve (\ref{eq:branch_max}) again.

\eit

In light of the above, we get the following algorithm for finding all GNEs in $\mc{K}_J$.

\begin{alg} \label{alg:pLME_J}
For the GNEP as in \reff{GNEP:Axi>=bi} and for a given $J\in\mc{P}$,
select a generic symmetric positive definite matrix $\Theta$
and a small positive value (say, $0.5$) for $\delta$.
Let $\mc{S}_J \coloneqq \emptyset$ and $j \coloneqq 1$.
Then, do the following:

\begin{itemize}

\item [Step~1]
Solve the optimization problem (\ref{eq:findfirstKKT}).
If it is infeasible, output the nonexistence of GNEs in $\mc{K}_J$ and stop.
Otherwise, solve (\ref{eq:findfirstKKT})
for a minimizer $u^{(1)}$ and go to Step~2.

\item [Step~2]
For each $i\in[N]$, compute the minimum value $\eps_i$ of
(\ref{eq:verifyGNE}) for $u\coloneqq u^{(j)}$.
If $\epsilon_i\ge0$ for every $i$, then update
$\mc{S}_J\coloneqq \mc{S}_J\cup \{ u^{(j)} \}$.

\item [Step~3]
Compute the maximum value $\theta_{\max}$ of \reff{eq:branch_max}.

\item [Step~4]
If $\theta_{\max} = \theta(u)$, then go to Step~5;
otherwise, let $\delta \coloneqq \delta/2$ and go to Step~3.

\item [Step~5]
Solve the optimization problem \reff{theta>=uj}.
If it is infeasible, output that $\mc{S}_J$ is the set of all GNEs in $\mc{K}_J$ and stop.
Otherwise, update $j\coloneqq j+1$ and solve \reff{theta>=uj} for a minimizer $u^{(j)} $,
then go to Step~2.

\end{itemize}
\end{alg}

If the GNEP is convex, every KKT point is a GNE,
so Step~2 can be skipped.
The properties of Algorithm~\ref{alg:pLME_J}
are summarized as follows.

\begin{prop}  \label{prop:iso_u}
For the GNEP as in \reff{GNEP:Axi>=bi},
the following properties hold for Algorithm~\ref{alg:pLME_J}:

\bit

\item [(i)]
If $\theta_{\max} = \theta(u^{(j)})$ and \reff{theta>=uj} is infeasible, then
$S_J$ is the set of all GNEs in $\mc{K}_J$.

\item [(ii)]
If $\theta_{\max} = \theta(u^{(j)})$ and $u^{(j+1)}$
is the minimizer of \reff{theta>=uj},
then $\dt$ satisfies \reff{range:dt}.

\item [(iii)]
Assume $u^{(1)}, \ldots, u^{(j)}$ are isolated points of $\mc{K}_J$.
Suppose $\Theta$ is a generic symmetric positive definite matrix,
then there exists $\delta>0$ such that $\theta_{\max} = \theta(u^{(j)})$,
i.e., $u^{(j)}$ is the maximizer of (\ref{eq:branch_max}).

\eit
\end{prop}
\begin{proof}
(i) When $\theta_{\max} =\theta(u^{(j)})$, the KKT point $u^{(j)}$ is the maximizer of \reff{eq:branch_max}.
If there is $v \in \mc{K}_J$
other than  $u^{(1)}, \ldots, u^{(j)}$, then
\[
\theta(v) > \theta(u^{(j)}) + \dt.
\]
On the other hand, when \reff{theta>=uj} is infeasible,
every $x \in \mc{K}_J$ must satisfy
\[
\theta(x)  < \theta(u^{(j)}) + \dt.
\]
Therefore, if $\theta_{\max} = \theta(u^{(j)})$ and
\reff{theta>=uj} is infeasible, then there are no KKT points in $\mc{K}_J$ except
$u^{(1)}, \ldots, u^{(j)}$.
This implies that all GNEs in $\mc{K}_J$ are contained in $\mc{S}_J$.

\medskip
(ii) If $\theta_{\max} = \theta(u^{(j)})$ and
$u^{(j+1)}$ exists, then as in (i), we can get
\[
\theta(u^{(j+1)}) > \theta(u^{(j)}) + \dt ,
\]
which means that \reff{range:dt} holds.

\medskip
(iii) For $\epsilon>0$, let $\mathbb{S}_{\epsilon}$ denote the set of
all $(n+1)$-by-$(n+1)$ symmetric positive definite matrices
whose largest eigenvalue equals one and
whose smallest eigenvalue is at least $\epsilon$.
The set of all $(n+1)$-by-$(n+1)$ symmetric positive definite matrices
of unit $2$-norm
is the union $\bigcup_{l=1}^{\infty} \mathbb{S}_{1/l}.$
For each $l\in \N$, we show that the conclusion holds for all
$\Theta \in \mathbb{S}_{1/l}$
except a set of Lebesgue measure zero.

Let $\Theta\in \mathbb{S}_{1/l}$ be an arbitrary matrix.
By the selection of $u^{(1)}, \ldots, u^{(j)}$, it holds
\[
\nu_1  \coloneqq  \theta( u^{(1)} )  <  \nu_2  \coloneqq  \theta( u^{(2)} )
<  \cdots <  \nu_j  \coloneqq   \theta( u^{(j)} )  .
\]
We consider the case $j>1$ for convenience
because the proof is almost the same for $j=1$.
When $\mc{K}_J$ has no other points except $u^{(1)}, \ldots, u^{(j)}$,
we have $\theta_{\max} = \theta(u^{(j)})$ for all $\dt > 0$.
So, we consider the opposite case and
suppose that $\bar{u}$ is a point in $\mc{K}_J$
that is different from $u^{(1)}, \ldots, u^{(j)}$,
and that $\nu_j\le\theta(\bar{u})$.
If $x$ is a minimizer of \reff{theta>=uj} in previous loops, then
\[
e_1^T\Theta e_1 + \| \bar{u} \|^2 \ge [\bar{u}]_1^T\Theta [\bar{u}]_1 \ge  [x]_1^T\Theta [x]_1  \ge e_1^T\Theta e_1 +\Vert x\Vert^2/l,
\]
with $e_1 = (1,0,\ldots, 0)^T$.
Thus, all minimizers of \reff{theta>=uj} in previous loops are contained in the ball
\[
B \coloneqq  \left\{ x\in\re^n: \Vert x\Vert^2 \le l \cdot  \| \bar{u} \|^2  \right\},
\]
which implies $u^{(k)}\in B$ for each $k$.
Recall the notation $[x]_d$ as in \reff{[x]d}.
Since each $u^{(k)}$ is an isolated point of $\mc{K}_J$, the set
\[
T_1  \coloneqq  \big \{ [x]_2: \, x\in\mc{K}_J \cap B,
    x \ne  u^{(k)}, 1 \le k \le j-1 \big \}
\]
is compact.
Since $\theta( x ) = \langle \theta, [x]_2 \rangle$,
we have
\[
\langle  \theta,  [u^{(1)}]_2 \rangle
<  \cdots <  \langle  \theta,  [u^{(j-1)}]_2 \rangle < \min_{ y \in T_1} \langle  \theta,  y \rangle .
\]
The right most minimization in the above is equivalent to
\be \label{min:cv(T1)}
\, \left\{
\begin{array}{cl}
\min &  \langle  \theta,  y \rangle\\
\st & y\in\mbox{conv}(T_1).
\end{array}
\right.
\ee
The convex hull $\mbox{conv}(T_1)$ is a compact convex set.
Observe that if \reff{min:cv(T1)} has more than one minimizer,
then $\theta$ is a singular normal vector of the convex body $\mbox{conv}(T_1)$.
The set of singular normal vectors of a convex body has Lebesgue measure zero.
This is shown in \cite[Theorem~2.2.11]{schneider2014convex}.
So, when $\Theta$ is generic in $\mathbb{S}_{1/l}$,
the linear optimization \reff{min:cv(T1)} has the unique minimizer $[u^{(j)}]_2$.
Let
$$T_2 \coloneqq T_1\setminus \{[u^{(j)}]_2\}.$$
Since $u^{(j)}$ is an isolated point of $\mc{K}_J$,
the set $T_2$ is also compact, so
\[
\langle  \theta,  [u^{(j)}]_2 \rangle
    <  \min_{ y \in T_2} \langle  \theta,  y \rangle .
\]
Then there must exist $\dt > 0$ such that
\[
\langle  \theta,  [u^{(j)}]_2 \rangle  +  \dt
<  \min_{ y \in T_2} \langle  \theta,  y \rangle .
\]
For the above $\dt$, we must have $\theta_{\max} = \theta(u^{(j)})$.
This means that the conclusion holds for all
$\Theta \in \mathbb{S}_{1/l}$
except for a set of Lebesgue measure zero,
for each $l\in \N$.
This completes the proof.
\end{proof}

When the cardinality $|\mc{K}_J|<\infty$,
all points in $\mc{K}_J$ are isolated,
so the following follows from Proposition~\ref{prop:iso_u}.
\begin{theorem} \label{tm:allinJ}
Consider the GNEP as in \reff{GNEP:Axi>=bi}.
For the given $J$, if $|\mc{K}_J|<\infty$,
then Algorithm~\ref{alg:pLME_J} returns all GNEs contained in $\mc{K}_J$
or detects their nonexistence.
\end{theorem}
When the GNEP is given by generic polynomials,
the critical set $\mc{K}$ (hence its subset $\mc{K}_J$ for each $J\in\mc{P}$) is finite; see \cite[Theorem~3.1]{nie2022algebraic}.
So, for generic GNEPs as in \reff{GNEP:Axi>=bi},
Algorithm~\ref{alg:pLME_J} can find all GNEs in $\mc{K}_J$ or detect their nonexistence.

\subsection{Finding all GNEs}
\label{ssc:allGNEs}

For a given $J$, Algorithm~\ref{alg:pLME_J} can
compute all GNEs in $\mc{K}_J$ or detect their nonexistence.
For the GNEP as in \reff{GNEP:Axi>=bi},
the set $\mc{P}$ is finite. By enumerating $J\in\mc{P}$,
we can get all GNEs or detect their nonexistence.
This gives the following algorithm.

\begin{alg} \label{alg:pLME_all}
For the GNEP as in \reff{GNEP:Axi>=bi},
formulate the label set $\mc{P}$.
Let $\mc{S} \,:=\, \emptyset$.
For each $J\in\mc{P}$, do the following:

\begin{itemize}

\item[Step~1]
Formulate the pLME  $\lmd_{i,J_i}$ as in (\ref{eq:pLME}) for each player $i$.

\item[Step 2]
Apply Algorithm~\ref{alg:pLME_J} to find the set $S_J$ of all GNEs in $\mc{K}_J$.

\item[Step 3]
Update $\mc{S}\,\coloneqq\, \mc{S} \cup S_J.$
\end{itemize}
\end{alg}

The following result follows from Theorem~\ref{tm:allinJ}.

\begin{theorem}\label{tm:all}
For the GNEP as in \reff{GNEP:Axi>=bi},
assume the critical set $\mc{K}$ is finite and $\Theta$ is a generic symmetric positive definite matrix.
Then, after enumerating all $J\in\mc{P}$,
Algorithm~\ref{alg:pLME_all} finds all GNEs if $\mc{S}\ne\emptyset$, or
detects nonexistence of GNEs if $\mc{S}=\emptyset$.
\end{theorem}

\begin{remark}
(i) Theorem~\ref{tm:all} implies the finite convergence of Algorithm~\ref{alg:pLME_all}.
That is, when the set $\mc{K}$ is finite and $\Theta$ is generically positive definite,
our method can get all GNEs by solving finitely many polynomial optimization subproblems.
Furthermore, as shown in Section~\ref{sc:pops}, the Moment-SOS hierarchy is guaranteed to find global minimizers for these subproblems, under some genericity assumptions.
Therefore, for generic GNEPs, our method can find all GNEs if there exist finitely many ones,
or detect their nonexistence, by solving finitely many semidefinite programs.

(ii) Computing all GNEs is a challenging task.
To the best of the authors' knowledge, there are few
efficient methods for doing this.
For GNEPs given by polynomial or rational functions, the semidefinite relaxation method in \cite{Nie2020nash,NieTangZgnep21} guarantees to find all GNEs, under the assumption that $|\mc{K}| < \infty$ and {\it feasible extensions} exist;
the polyhedral homotopy method in \cite{Lee2023} finds all GNEs if the number of roots for the complex KKT system equals the mixed volume (with the origin included in the support).
In contrast, Algorithm~\ref{alg:pLME_all} does not require such conditions to find all GNEs.

(iii) The assumption that $\mc{K}$ is a finite set holds when every $f_i$ and $b_i$ are given by generic polynomials and each $A_i$ is a generic matrix.
This is shown in \cite{nie2022algebraic}.
However, it is possible that $|\mc{K}|$ is infinite when the GNEP is not generic; see Example~\ref{ex:shared}.
For these cases, if the GNEP is convex, then our method still guarantees finding at least one GNE, but it may fail to find all of them;
otherwise, our method may or may not find GNEs and check the completeness of the solution set.
Besides that, the assumption that $\Theta$ is a generic symmetric positive definite matrix is independent of the problem.
Indeed, Theorem~\ref{tm:all} implies that if we set $\Theta \coloneqq RR^T$
with a randomly generated $R\in \re^{(n+1)\times (n+1)}$ whose entries obey the normal distribution, then Algorithm~\ref{alg:pLME_all} can almost always find all GNEs as long as $\mc{K}$ is finite.

(iv) To find all GNEs or detect their nonexistence,
Algorithm~\ref{alg:pLME_all} requires enumerating all
$J\in\mc{P}$, where $|\mc{P}| = |\mc{P}_1|\times \cdots\times |\mc{P}_N|$.
For each $i$, the cardinality of $|\mc{P}_i|$ is at most $\binom{m_i}{r_i}$, which
may increase at a polynomial (or even linear) rate for many interesting cases.
For instance, if $m_i = r_i+1$, then $\binom{m_i}{r_i} = m_i$ and the growth rate is linear.
Moreover, even if $|\mc{P}_i|$ increases exponentially as the problem size increases,
the cardinality of $\mc{P}_i$ is usually less than $\binom{m_i}{r_i}$ (e.g., consider the boxed constraint $L_{i,j} \le x_{i,j} \le U_{i,j}$ with scalars $L_{i,j} < U_{i,j}$ for each $i,j$).
Last, if we aim to find a single GNE, then we typically do not need to enumerate all $J\in\mc{P}$, as shown in Section~\ref{sec:NumExp}.
\end{remark}

To end this section, we discuss a practical strategy based on constraint screening to reduce the number of sets $J$ for enumeration in computational practice.
If one has a priori information that some constraints cannot be active simultaneously at GNEs,
then all $J\in\mc{P}$ such that the labels of these constraints are contained in $J$ can be precluded from the enumeration.
For instance, suppose that $n_1=n_2$ and consider the two player GNEP with joint box constraints:
\be\label{eq:joint_box_Xi}  X_i(\xmi) = \left\{ x_i\in \re^{n_i}
\left| \begin{array}{cl}
0 \le x_{1,j} + x_{2,j} \le 1,& \forall j\in[n_i],\\
\hat{A}_i\xpi \ge \hat{b}_i(\xmi) &
\end{array}
\right. \right\}.   \ee
For every $j=1\ddd n_1$,
let $\lmd_{i,2j-1}$ and $\lmd_{i,2j}$ be the Lagrange multipliers associated to $x_{1,j} + x_{2,j} \ge 0$ and $ 1- x_{1,j} - x_{2,j} \ge 0$, respectively.
Since these two constraints cannot be simultaneously active, at least one of
$\lmd_{1,2j-1}$ and $\lmd_{2,2j}$ must be zero.
Similarly, at most one of $\lmd_{2,2j-1}$ and $\lmd_{1,2j}$ can be nonzero.
Thus for this GNEP, if we let
\[ \mc{Q} \coloneq \left\{ J\in\mc{P} \left|
\begin{array}{c}
\mbox{$\exists$ $i\in [2]$ and $j\in [n_1]$ such that}\\
2j-1\in J_i,\ 2j\in J_{3-i}
\end{array}
\right. \right\},  \]
then one may skip all $J\in \mc{Q}$ for finding GNEs.
To illustrate the impact, if we momentarily disregard the constraints $\hat{A}_i\xpi \ge \hat{b}_i(\xmi)$, then the cardinality $|\mc{P}| = 4^{n_1}$, whereas the number of valid sets $J$ to enumerate becomes $|\mc{P}\setminus \mc{Q}| = 2^{n_1}$.

\section{Solving Polynomial Optimization}
\label{sc:pops}

We now show how to solve polynomial optimization problems
that appear in Algorithms~\ref{alg:pLME_J} and \ref{alg:pLME_all}.
They can be generally expressed in the form:
\be    \label{eq:pop}
\left\{\begin{array}{rcl}
f_{min} \coloneqq  & \min\limits_{z} & f(z)\\
& \st & p(z) = 0 \,\, (\forall p \in \Phi),\\
& & q(z) \geq 0 \,\, (\forall q \in \Psi ),
\end{array}
\right.
\ee
where the variable $z$ represents either $x\in\re^n$ or $x_i\in\re^{n_i}$ for the $i$th player, and $\Phi$, $\Psi$ are finite sets of equality and inequality constraining polynomials, respectively.
The Moment-SOS relaxations are efficient
for solving \reff{eq:pop} globally.
We refer to the books \cite{HenrionKordaLasserre2020,lasserre2015,Lau09,MPO} for a more detailed introduction.

Denote the degrees
\begin{equation*}
\baray{rcl}
d_0 & \coloneqq & \max \big\{\left\lceil \deg(p)/2 \right\rceil  : \,
p \in   \Phi \cup \Psi  \big \},  \\
d_1 & \coloneqq & \max \big\{\left\lceil \deg(f)/2 \right\rceil, d_0 \big \}.
\earay
\end{equation*}
Let $\ell$ be the dimension of $z$. For a degree $k \geq d_1$,
the $k$th order moment relaxation for solving \reff{eq:pop} is
\begin{equation}\label{eq:mom}
\left\{\begin{array}{rcl}
f_{mom,k} \coloneqq  & \min & \langle f, y \rangle  \\
& \st & y \in \mathbb{R}^{\mathbb{N}_{2k}^{\ell}},\, y_0 = 1, L_p^{(k)}[y] = 0 \, (p \in \Phi),  \\
& & M_k[y] \succeq 0, \, L_q^{(k)}[y] \succeq 0 \, (q \in \Psi ).
\end{array}
\right.
\end{equation}
The dual optimization of \reff{eq:mom} is the $k$th order SOS relaxation
\begin{equation}\label{eq:sos}
\left\{\begin{array}{rcl}
f_{sos,k} \coloneqq  & \max & \gamma\\
& \st & f - \gamma \in \textrm{Ideal}[\Phi]_{2k}+\textrm{QM}[\Psi ]_{2k}.
\end{array}
\right.
\end{equation}
We refer to Section~\ref{sec:pre} for the notation $\langle f, y \rangle$, $L_p^{(k)}[y]$, $M_k[y]$, $\textrm{Ideal}[\Phi]_{2k}$, $\textrm{QM}[\Psi ]_{2k}$ in the above.
It is worth remarking that \reff{eq:mom}-\reff{eq:sos} is a primal-dual pair of semidefinite programs.
For $k = d_1, d_1+1, \ldots$, the primal-dual pair \reff{eq:mom}-\reff{eq:sos} is called the Moment-SOS hierarchy.
Its convergence property can be summarized as follows.
When $\textrm{Ideal}[\Phi]+\textrm{QM}[\Psi ]$ is archimedean, we have $f_{mom,k} \to f_{min}$ as $k\to\infty$.
Moreover, if the linear independence constraint qualification, strict complementarity condition, and second order sufficient optimality conditions hold at every minimizer, then $f_{sos,k} =  f_{min}$ for all $k$ that is big enough (see \cite{MPO,nie2014optimality}).

In the following, we show how to extract minimizers for
\reff{eq:pop} from the moment relaxation.
Suppose $y^{(k)}$ is a minimizer of \reff{eq:mom}.
If $y^{(k)}$ satisfies the flat truncation:
there exists a degree $t \in [d_1, k]$ such that
\be   \label{eq:rank}
\rank \, M_t[y^{(k)}]  \,= \, \rank \, M_{t-d_0}[y^{(k)}] ,
\ee
then $f_{min} = f_{mom,k}$
and we can extract $r\coloneqq  \rank\,M_t[y^{(k)}]$ minimizers
for (\ref{eq:pop}) (see \cite{HL05,nie2013certifying,MPO}).
Indeed, flat truncation is a sufficient and almost necessary condition
for extracting minimizers. This is shown in \cite{nie2013certifying}.

The Moment-SOS algorithm for solving (\ref{eq:pop}) is as follows.

\begin{alg}\label{alg:mom-sos}
For the polynomial optimization \reff{eq:pop}, initialize $k \coloneqq d_1$.
\begin{itemize}

\item[Step~1] Solve the moment relaxation \reff{eq:mom}.
If it is infeasible, then \reff{eq:pop} is infeasible and stop.
Otherwise, solve it for a minimizer $y^{(k)}$.

\item[Step~2] Check whether or not $y^{(k)}$ satisfies the rank condition \reff{eq:rank}.
If \reff{eq:rank} holds,
then extract $r\coloneqq \rank\,M_t[y^{(k)}]$ minimizers of \reff{eq:pop} and stop.
Otherwise, let $k \coloneqq k+1$ and go to Step 1.

\end{itemize}

\end{alg}

Algorithm \ref{alg:mom-sos} can be implemented
in the software \texttt{GloptiPoly3} \cite{GloPol3},
which calls SDP package such as {\tt MOSEK} \cite{mosek}.
For Algorithms~\ref{alg:pLME_J} and \ref{alg:pLME_all},
the optimization problem~\reff{eq:pop} is one of \reff{eq:findfirstKKT},
\reff{eq:verifyGNE}, \reff{theta>=uj}, or \reff{eq:branch_max}.
We have the following remarks:
\begin{itemize}

\item
For the minimization problem \reff{eq:findfirstKKT} and \reff{theta>=uj},
we have $z \coloneqq x$ and $f(x)\coloneqq \theta(x)$,
where $\theta(x)$ is defined by the generically selected positive definite matrix $\Theta$.
So, if they are feasible, then they have a unique optimizer.
Moreover, the equality constraints of both \reff{eq:findfirstKKT}
and \reff{theta>=uj} define finite real varieties	
when the polynomials for the GNEP have generic coefficients (see \cite{nie2022algebraic}).
In these cases, flat truncation (\ref{eq:rank}) holds with
$r=1$ for all $k$ that is big enough \cite{Nie13FiniteVarieties}.

\item
For the polynomial optimization problem (\ref{eq:verifyGNE}) of verifying GNEs,
we have $z\coloneqq x_i$ and $f(x_i)\coloneqq f_i(x_i,u_{-i}) - f_i(u_i,u_{-i})$.
This problem must be feasible, as $u_i$ is a feasible point.
If $f_{mom,k} \ge 0$, we can terminate Algorithm~\ref{alg:mom-sos}
directly since we do not need to extract minimizers for this case.

\item
For the maximization problem \reff{eq:branch_max},
we have $z \coloneqq x$ and $f(x)\coloneqq -\theta(x)$.
It is always feasible since $u^{(j)}$ is a feasible point.
Furthermore, when the GNEP is given by generic polynomials,
the equality constraints of \reff{eq:branch_max} give a finite real variety,
so flat truncation (\ref{eq:rank}) holds for all $k$ that is big enough.

\end{itemize}

Recall that $e_i$ represents the vector of all zeros except that
$i$th entry is $1$. The notation $y^{(k)}_{e_i}$
denotes the entry of $y^{(k)}$  labeled by $e_i$.
The following is the convergence property of Algorithm~\ref{alg:mom-sos}
when it is applied to solve polynomial optimization problems \reff{eq:findfirstKKT}, \reff{theta>=uj}, or \reff{eq:branch_max}.
These conclusions are shown in \cite{nie2023convex, Nie2020nash}.

\begin{theorem}
Suppose the optimization problem~\reff{eq:pop}
is \reff{eq:findfirstKKT}, \reff{theta>=uj}, or \reff{eq:branch_max}.
Assume $\Theta$ is a generic symmetric positive definite matrix
and the real variety of $\Phi$ is a finite set.
Then, we have:

\begin{itemize}

\item[(i)]
If \reff{eq:pop} is infeasible,
then the moment relaxation \reff{eq:mom}
must be infeasible when the order $k$ is large enough.

\item[(ii)]
Suppose \reff{eq:pop} is feasible.
Then $f_{mom,k} = f_{min}$ and the flat truncation holds
for all $k$ that is big enough.
Furthermore, if the optimization problem~\reff{eq:pop} is
\reff{eq:findfirstKKT} or \reff{theta>=uj},
then $u^{(k)} \coloneqq (y^{(k)}_{e_1},y^{(k)}_{e_2}, \ldots, y^{(k)}_{e_n})$ is the unique minimizer of \reff{eq:pop},
when the order $k$ is large enough.

\end{itemize}
\end{theorem}

\section{Numerical experiments}
\label{sec:NumExp}

This section presents the numerical experiments of applying Algorithms~\ref{alg:pLME_all} for solving GNEPs with quasi-linear constraints.
In computations, involved polynomial optimization problems are solved globally with Algorithm~\ref{alg:mom-sos},
using the \MATLAB software \texttt{GloptiPoly3} \cite{GloPol3}.
Semidefinite programs are solved using the
\texttt{MOSEK} solver \cite{mosek} with \texttt{Yalmip} \cite{Yalmip}.
The computations were implemented using \MATLAB R2023b on a laptop equipped with a
12th Gen Intel(R) Core(TM) i7-1270P 2.20GHz CPU and 32GB RAM.
The computational results are reported with four decimal places.
For convenience of expression, the constraints are ordered from left to right and from top to bottom in each problem.

\begin{example} We apply Algorithm \ref{alg:pLME_all} to solve some GNEPs from the existing literature. These problems are explicitly given in Appendix~\ref{sc:existing_refs}, each accompanied by its corresponding label.
We report our numerical results in Table \ref{table:appendix}.
The notation $\#x^*$ stands for the number of computed GNEs.
\begin{table}[htb!]
    \centering
    \caption{Numerical results for GNEPs in Appendix~\ref{sc:existing_refs}}
    \label{table:appendix}
    \small
    \begin{tabular}{|l|c|c|}
    \hline
       Problem  & {\small \#$x^*$} & All GNEs $x^* = (x_1^*,x_2^*, \ldots, x_N^*)$ \\
       \hline
        FKA3 & 3 & {\small
        \begin{tabular}{l}
        $ (-0.3805, -0.1227, -0.9932),\, (0.3903, 1.1638),\, (0.0504, 0.0176)$;\\
        $ (-0.8039, -0.3062, -2.3541),\, (0.9701 , 3.1228),\, (0.0751, -0.1281)$;\\
       $(1.9630, -1.3944, 5.1888),\, ( -3.1329, -10.0000),\, (-0.0398, 1.6392)$\end{tabular}}\\ \hline
        FKA4 & 1 & $ (1.0000,1.0000,1.0000),\, (1.0000,1.0000),\, (1.0000,1.0000)$\\ \hline
        FKA8 & 2 & $(0.3333),\,(0.5000),\,(0.6667)$ and  $ (0.0000),\,(1.0000),\, (1.0000)$\\ \hline
        NT59 & 1 & $ (0.7000,0.1600),\, (0.8000,0.1600),\, (0.8000,0.4700)$\\ \hline
        NT510 & 1 & $ (1.7184),\, (1.8413,0.6700),\, (1.2000,0.0823,0.0823)$\\ \hline
				NTGS53 & 2 & \begin{tabular}{l}
				$ (0.0000, 0.5000),  (0.5000, 0.0000);$\\
				$(0.0000, 0.5000),     (0.0000, 0.5000)$
				\end{tabular}\\ \hline
        NTGS54 & 1 & $ (0.1000, 0.4000), (0.1000, 0.4000)$\\ \hline
				FR33 & 4 & \begin{tabular}{l}
				$ (0.0000, 2.0000),\,(0.0000,6.0000)$;\\
				$(0.0000, 0.0000),\, (0.0000, 0.0000)$;\\
				$(1.1876, 1.9062),\, (1.2481, 0.0000)$;\\
				$(1.0000, 2.0000),\, (1.0000, 2.0000)$
				\end{tabular}\\ \hline
        SAG41 & 1 & $ (0.5588,0.5588),  (0.2647,0.2647)$\\ \hline
    \end{tabular}
\end{table}
\end{example}

\begin{example}
\label{ex:0909_0324}
Consider the $2$-player convex GNEP:
\begin{gather*} \text{F}_1(x_{-1}) :\left\{\begin{array}{cl}
\min\limits_{x_1\in\re^4} &\displaystyle (x_{1,1}-1)^2+x_{2,4}(x_{1,2}-1)^2+(x_{1,3}-1)^2\\
& \qquad \qquad \qquad \qquad +(x_{1,4}-2)^2+(\mathbf{1}^Tx_2-1)\mathbf{1}^Tx_1\\
\st & 0 \leq x_1 \leq x_2,\,\,  x_{2,3}(x_{2,3}-1)(x_{2,3}-3) \geq x_{1,4},
\end{array}
\right.
\\
\text{F}_2(x_{-2}) :\left\{\begin{array}{cl}
\min\limits_{x_2\in\re^4} & x_{1,1}x_{2,1}^2-x_{2,2}+x_{1,3}(x_{2,3}-1)^2+x_{1,4}(x_{2,4}+1)^2\\
\st & x_{2,1}- x_{2,2} - x_{1,2} \ge 0,\, 2x_{1,1} - x_{2,1} + x_{2,2} \ge 0,\\
	& x_{2,1} + x_{2,2} + x_{1,1} + x_{1,2}\ge 0,\\
  & 4x_{1,1} - 2x_{1,2}-x_{2,1} - x_{2,2}\ge 0,\,\, x_{2,3}\ge 0,\\
	&   x_{1,3}(3x_{1,3}-1)(x_{1,3}-1)\ge 3x_{2,4},\  3 \geq x_{2,3} +x_{2,4} .
\end{array}
\right.
\end{gather*}
There are 288 $J\in \mc{P}$.
It took around 4449.10 seconds to find all GNEs by Algorithm~\ref{alg:pLME_all}
and the computational time of each $\mc{K}_J$ is between 0.34-147.30 seconds.
The first GNE was found in 149.09 seconds, and we found 6 GNEs in total from 48 $\mc{K}_J$'s,
which are
\[
\begin{array}{|c|} \hline
x_1^* = (0.3333, 0.0000, 0.3333, 0.0000),\quad
x_2^* = (0.6667, 0.6667, 1.0000, 0.0000);\\ \hline
x_1^* = (0.0000, 0.0000, 0.0000, 0.0000),\quad
x_2^* = (0.0000, 0.0000, 3.0000, 0.0000);\\  \hline
x_1^* = (0.5000, 0.0000, 0.0000, 0.0000),\quad
x_2^* = (1.0000, 1.0000, 0.0000, 0.0000);\\  \hline
x_1^* = (0.0000, 0.0000, 1.0000, 0.0000),\quad
x_2^* = (0.0000, 0.0000, 1.0000, 0.0000);\\  \hline
x_1^* = (0.7071, 0.0000, 0.0000, 0.0000),\quad
x_2^* = (0.7071, 0.7071, 0.0000, 0.0000);\\   \hline
x_1^* = (0.0000, 0.0000, 0.0000, 0.0000),\quad
x_2^* = (0.0000, 0.0000, 0.0000, 0.0000). \\  \hline
\end{array}
\]
In particular, we found $5$ GNEs in $\mc{K}_J$ for
\[\begin{aligned}
    J \,=\,& \big(\, \{2, 5, 7, 8\},\, \{1, 4, 5, 6\}\,\big),
    \big(\, \{2, 5, 7, 8\}, \{1, 4, 5, 7\}\,\big),  \big(\, \{2, 5, 7, 8\},\, \{1, 4, 6, 7\}\,\big),\\
    & \big(\, \{2, 5, 7, 9\},\, \{1, 4, 5, 6\}\,\big),
    \big(\, \{2, 5, 7, 9\},\, \{1, 4, 5, 7\}\,\big),
    \mbox{or}\,\big(\, \{2, 5, 7, 9\}, \{1, 4, 6, 7\}\,\big).
\end{aligned}\]

\end{example}

\begin{example}\label{ex:0802_0313}
Consider the 2-player nonconvex NEP
\begin{equation*}
\text{F}_i(x_{-i}) :
\left\{\begin{array}{cl}
\min\limits_{x_i \in \mathbb{R}^{n_i}} & f_i(x_i, x_{-i}) \\
\st  & A_i x_i \geq b_i,
\end{array}\right.
\end{equation*}
where $n_1 = 7$, $n_2 = 5$, and
\begin{gather*}
  		f_1(x) = 3x_{1,1}^2+4x_{1,2}^2+4x_{1,2}x_{2,1}+3x_{1,4}x_{2,4}+4x_{1,6}x_{2,4},\\
     f_2(x) = x_{1,2}x_{2,2}+3x_{1,5}x_{2,4}+x_{1,6}x_{2,2}+x_{2,1}^2+2x_{2,1}x_{2,2}+x_{2,3}^2,\\
     A_1 = \left[\begin{array}{rrrrrrr}
     0  &  -3  &   0  &   2  &   3   &  1   &  3\\
        2&    -1 &    2 &   -2 &    1 &    1  &  -2\\
         -1 &   -1 &    0 &    2 &    2 &    1 &   -3\\
        1  &   1  &   0  &   1   &  0   & -1  &   2\\
         1  &   2  &   0  &   2  &  -3  &  -2  &  -2\\
       -1   &  0   & -2  &   3   &  1   & -1  &  -3\\
        0  &  -1  &  -3  &  -2  &  -2  &  -3   &  2\\
         -3 &    2  &   0  &   1 &   -3  &  -2  &  -3\\
         1   &  1   &  1   &  2   &  3   &  0    & 1
     \end{array}\right],\,\,
     A_2 = \left[\begin{array}{rrrrr}
    2  &  -3   & -1   & -1   & -1\\
        -3 &    4 &    3 &    2 &   -3\\
        1   &  2   &  1  &   0  &   2\\
         2  &  -3   &  2  &   3  &  -1\\
        -3   &  1   &  2   &  2   &  2\\
          2  &   1  &  -2  &  -3   &  4\\
          0   &  2  &   3   &  1   &  2
     \end{array}\right],\\
     b_1 = \begin{bmatrix}
 	1 &
         5 &
         4 &
         2 &
         2 &
         2 &
         2 &
        1 &
         -1
     \end{bmatrix}^T,\,\,\,\quad
      b_2 = \begin{bmatrix}
    1 &
         3 &
         1 &
         1 &
         3 &
         0 &
         -1\end{bmatrix}^T. 
 \end{gather*} 
There are a total of 756 $J\in \mc{P}$.
It took around 1774.40 seconds to find all NEs by Algorithm~\ref{alg:pLME_all}
and the computational time for each $\mc{K}_J$ is between 1.29-12.51 seconds.
The first NE was detected within 151.88 seconds, which is
\begin{align*}
 &   x_1^* = (1.7344,-1.2108,0.8670,0.9041,0.9669,-3.2800,-1.3538),\\
&    x_2^* = (-0.4706, -1.0941, 4.4392,-3.3294,0.2314).
\end{align*}
This is the unique NE and the unique KKT point. It is contained in $\mc{K}_J$ for
\[\begin{aligned}
    J \,=\,& \big(\, \{1,2,3,4,5,7,8\},\, \{1,2,4,5,6\}\,\big), \quad
    \big(\, \{1,2,3,4,6,7,8\},\, \{1,2,4,5,6\}\, \big), \\
    &   \mbox{or}\quad \big(\, \{1,2,3,4,7,8,9\},\, \{1,2,4,5,6\}\,\big).
\end{aligned}\]
\end{example}

\begin{example}
\label{ex:D68_0305}
Consider the 2-player nonconvex GNEP
\begin{equation*}
\text{F}_i(x_{-i}) :
\left\{\begin{array}{cl}
\min\limits_{x_i \in \mathbb{R}^{n_i}} & (-1)^i\|x_1+\mathbf{1}\|^2 + (-1)^{i+1}\|x_2+\mathbf{1}\|^2 \\
\st  & \alpha_{i,1}^Tx_1 + \beta_{i,1}^Tx_2 + \gamma_{i,1} \geq 0, \\
& \alpha_{i,2}^Tx_1 + \beta_{i,2}^Tx_2 + \gamma_{i,2} \geq 0, \\
& 1 \geq \mathbf{1}^T x_i,\,\, x_i \geq 0,
\end{array}\right.
\end{equation*}
where $n_1 = 4$, $n_2 = 2$, and $\gamma_{1,1} = 2,\,\,\gamma_{1,2} = 4,\,\,\gamma_{2,1} = -1,\,\,\gamma_{2,2} = 4$,
\begin{gather*}
    \alpha_{1,1} = \left[\begin{array}{r}
        5\\-4\\-4\\1
    \end{array}\right], \,\,\alpha_{1,2} = \left[\begin{array}{r}
        5\\5\\5\\-3
    \end{array}\right],\,\, \alpha_{2,1} = \left[\begin{array}{r}
        -2\\2\\-4\\5
    \end{array}\right],\,\,\alpha_{2,2} = \left[\begin{array}{r}
        0\\-4\\-5\\-2
    \end{array}\right],
\\
    \beta_{1,1} = \left[\begin{array}{r}
        4\\-2
    \end{array}\right], \,\,\beta_{1,2} = \left[\begin{array}{r}
        5\\5
    \end{array}\right],\,\, \beta_{2,1} = \left[\begin{array}{r}
        -5\\3
    \end{array}\right],\,\,\beta_{2,2} = \left[\begin{array}{r}
        -2\\-5
    \end{array}\right].
\end{gather*}
There are a total of 280 $J \in \mc{P}$.
It took around 171.49 seconds to find all GNEs by Algorithm~\ref{alg:pLME_all}
and the computational time for each $\mc{K}_J$ is between 0.06-24.84 seconds.
The first GNE was detected within 127.56 seconds, which is
\begin{equation*}
x_1^* = (0.0000, 0.0000, 0.0000, 1.0000),\quad
x_2^* = (0.8387, 0.0645).
\end{equation*}
This is the unique GNE. It is contained in $\mc{K}_J$ for $J = \big(\{3,4,5,6\},\{1,2\}\big)$.
We remark that this problem has seven KKT points.
\end{example}

\begin{example}
\label{ex:0324_0129}
Consider the 2-player nonconvex GNEP
\begin{equation*}
\text{F}_i(x_{-i}) : \left\{\begin{array}{cl}
\underset{x_i \in \mathbb{R}^{3}}{\text{min}} & f_i(x) \\
\st  & A_i x_i \geq B_i[x_{-i}]_1 + d_i(x_{-i}),
\end{array}\right.
\end{equation*}
where
\begin{gather*}
    f_1(x) = x_{1,1}(x_{1,1}-2x_{2,1}) + x_{1,2} \cdot \mathbf{1}^Tx_2, \quad
    f_2(x) = x_{2,1}(2+2x_{2,2}) + x_{2,3} \cdot \mathbf{1}^Tx_1,
\\
    A_1 = \left[\begin{array}{rrr}
        -4&    -1&    -2\\
         0 &    3 &    4\\
        -1  &   5  &  -3\\
         5   & -1   & -3\\
         5    &-4    & 0\\
         0     &4    &-5\\
        -3     &4    &-5
    \end{array}\right],\,\,\,
B_1 = \left[\begin{array}{rrrr}
     0&    -5&    -4&   -2\\
    -1 &    3 &    6 &   -1\\
    -6  &  0  &  2  &   0\\
    -5   &  2   &  3   &  0\\
    -5&     3    & 0    & 5\\
     0 &    0    &-1     &3\\
     3  & -1 &  0   &  0
    \end{array}\right],\,\,\, d_1 = \left[\begin{array}{c} x_{2,3}^2 \\ x_{2,2}^2\\0\\ -x_{2,2}^2 \\0\\0\\0\end{array}\right],
\\
A_2 = \left[\begin{array}{rrr}
    -5&     2&    -1\\
     2 &    3 &    2\\
     1  &  -1  &   3\\
     5   &  0   & -2\\
    -3    &-4    & 1\\
     5     &4    &-1\\
    -3     &4    & 0
    \end{array}\right],\,\,\,
B_2 = \left[\begin{array}{rrrr}
    -5&     6&     4&    0\\
    -4 &   -2 &    4 &   5\\
    -4  &  3  &   6  & -1\\
     0   &  6   & -1   &  4\\
    -4   &-1    &-3    & 3\\
    -1    & 3    &-4    &-2\\
    -1    & 2    &0    &0
    \end{array}\right],\,\,\, d_2 = \left[\begin{array}{c} 0 \\ 0\\0\\ x_{1,3}^2 \\0\\-x_{1,2}^2\\0\end{array}\right].
    \end{gather*}
\noindent
There are a total of 1225 $J \in \mc{P}$. Five of them contain GNEs.
It took around 192.54 seconds to find all GNEs by Algorithm~\ref{alg:pLME_all}
and the computational time for each $\mc{K}_J$ is between 0.10--2.40 seconds.
The first GNE was detected within 116.59 seconds.
We found 2 GNEs from 5 $\mc{K}_J$'s in total, which are
\begin{equation*}
\begin{array}{ll}
x_1^* = (0.2075,0.7518,-0.0779),  & x_2^* = (0.2258,0.4260,0.4706); \\
x_1^* = (-0.3079,0.7901,0.1566), & x_2^* = (-0.3011,0.7604,0.2406).
\end{array}
\end{equation*}
We remark that these GNEs are also the only two KKT points for this problem.
\end{example}

\begin{example}\label{ex:random}
Consider the convex GNEP with the $i$th player's optimization
\be \label{rand:cvxGNEP}
\text{F}_i(x_{-i}) : \left\{\begin{array}{cl}
\min\limits_{x_i \in \mathbb{R}^{n_i}} &  f_i(x_i, x_{-i}) \\
\st & A_ix_i \geq b_i(x_{-i}) \coloneqq b_i-B_ix_{-i},
\end{array}\right.
\ee
where the vector $b_i(x_{-i})$ has length $m_i$
and the objective $f_i$ is in the form
\[
f_i(x_i, x_{-i})  \, = \, c_i^Tx + x^TG_ix + (x^{[2]})^T H_i x^{[2]}.
\]
In the above, $c_i \in \re^n$, $G_i \in \re^{n \times n}$
is symmetric positive semidefinite and $H_i \in \re^{n \times n}$
is symmetric positive semidefinite with nonnegative entries.
For the above choice, $f_i$ is a convex
polynomial function in $x$ (see \cite[Example~7.1.4]{MPO}).
We use the \MATLAB function \texttt{unifrnd} to generate random matrices $A_i, b_i$,
and $B_i$ for GNEPs of different sizes. 
We generate the convex polynomial $f_i$ randomly as
$c_i = \mathtt{randn}(n,1)$,
$G_i = R_1^T R_1$ with $R_1 = \mathtt{randn}(n)$, and
$H_i = R_2^T R_2$ with $R_2 = \mathtt{rand}(n)$.
We create 10 instances for each case and solve them with Algorithm~\ref{alg:pLME_all}.
The computational results are reported in Table \ref{table:random}.
For each instance, $\#(\mc{K}_J \ne \emptyset)$ counts the number of $\mc{K}_J$
that contains at least one KKT point and $\#\mbox{GNEs}$ counts the number of all GNEs.
The ``Avg. Time of Alg.~\ref{alg:pLME_J} for a single $\mc{K}_J$''
gives the average time (in seconds) taken by Algorithm \ref{alg:pLME_J}.
\begin{table}[ht]
\caption{Computational results for randomly generated convex GNEPs as in \reff{rand:cvxGNEP}.}
\label{table:random}}
{\tiny
\[\begin{array}{|c|c|c|c|c|c|}
\hline
N & \begin{array}{c}
    (n_1,\ldots,n_N)\\
    (m_1,\ldots,m_N)
\end{array} & |\mc{P}| & \mbox{\#}(\mc{K}_J \ne \emptyset) & \mbox{\#GNEs} & \begin{array}{c}
\mbox{Avg. Time of Alg.~\ref{alg:pLME_J}}\\
\mbox{for a single } \mc{K}_J
\end{array}\\ [5pt]
\hline
\multirow{4}{*}{2}
& \begin{array}{c}
  (3,4)\\
  (6,6)
\end{array} & 300 & \begin{array}{c} 300,300,300,300, 300,\\300, 300, 300, 300, 283\end{array}&  \begin{array}{c}1 ,1,1, 1, 1, \\ 1, 1,1, 1, 1\end{array} &  \begin{array}{c} 11.8,14.8, 14.7, 16.9,20.0,\\21.7, 12.6, 14.5, 12.6, 8.0 \end{array}\\ [5pt] \cline{2-6}
& \begin{array}{c}
  (3,3)\\
  (6,6)
\end{array} & 400 & \begin{array}{c} 400, 400, 400, 400, 400,\\ 400, 400, 380, 400 ,396\end{array}&  \begin{array}{c}3,1,2,1,1,\\ 1,2,1,1,1\end{array} &  \begin{array}{c}7.5, 8.7,12.0, 8.7, 9.2\\ 10.1,13.6, 11.6,12.9,13.4\end{array}\\ [5pt] \hline
\multirow{4}{*}{3}& \begin{array}{c}
  (2,3,3)\\
  (4,4,4)
\end{array} & 96 & \begin{array}{c}73, 96, 72, 96, 96, \\ 72, 96, 96, 96,96 \end{array}&  \begin{array}{c}3, 3, 2, 1, 1, \\ 2, 1, 1,1,1\end{array} &  \begin{array}{c}25.9, 31.2, 22.9, 23.0, 22.8, \\ 22.9, 30.6, 25.7, 15.7, 26.4\end{array}\\ [5pt]\cline{2-6}
& \begin{array}{c}
  (2,2,3)\\
  (4,4,5)
\end{array} & 360 & \begin{array}{c} 360,334, 360,360, 360\\360, 360,360,360,326\end{array}&  \begin{array}{c} 2,3,2,2,1\\ 1,1,2,1,1\end{array} &  \begin{array}{c} 16.6,14.9,19.6, 19.6,20.7\\26.8,24.1,24.4,17.8,25.5\end{array}\\[5pt] \hline
\multirow{3}{*}{4}& \begin{array}{c}
  (1,1,1,1)\\
  (3,3,3,3)
\end{array} & 81 & \begin{array}{c} 81, 81, 81, 81, 81,\\81, 81, 81, 81, 81\end{array}&
\begin{array}{c}1, 5, 1, 3, 1 \\ 3,1, 1, 1,1\end{array} &  \begin{array}{c} 2.5, 5.8, 1.9, 9.8, 2.4,\\ 2.7, 3.8, 2.7, 2.5, 3.3\end{array}\\ [5pt]\cline{2-6}
& \begin{array}{c}
  (1,2,2,2)\\
  (4,4,4,4)
\end{array} & 864 & \begin{array}{c} 863,864,841,864,864,\\864,432,864,864,864\end{array}&
\begin{array}{c}2,1,1,1,1, \\ 3,1,1,1,1\end{array}&   \begin{array}{c}9.3,12.4,20.0,25.5,11.1,\\13.0,18.6,22.8,23.2,17.6\end{array}\\ [5pt] \hline
\multirow{3}{*}{5}& \begin{array}{c}
  (1,1,1,1,1)\\
  (3,3,3,3,3)
\end{array} & 243 & \begin{array}{c}81,243,243,243,243,\\243,243,243,243,243\end{array}&  \begin{array}{c} 1,3,1,3,3,\\3,3,1,3,1\end{array} &  \begin{array}{c}1.2,2.6,2.9,2.9,3.3,\\3.7,4.3,5.5,7.2,2.5\end{array}\\ [5pt]\cline{2-6}
& \begin{array}{c}
  (1,1,1,1,1)\\
  (3,4,4,4,4)
\end{array} & 768 & \begin{array}{c} 195,768,768,768,768, \\ 768,768,768,768,192\end{array}&
\begin{array}{c}2,3,3,1,1, \\ 3,1,1,4,1\end{array}&  \begin{array}{c}1.3,4.1,15.4,7.6,8.9, \\ 10.2,12.1,16.6,94.4,10.1\end{array}\\ [5pt] \hline
\end{array}\]
\end{table}
Besides that, to visually illustrate the computational scalability of our method, we present in Figure~\ref{fig:time_analysis} for the relation between the computational time and the problem sizes of the GNEP in (\ref{rand:cvxGNEP}) for $N=2$.
\begin{itemize}
\item Figure~\ref{fig:time_analysis}(a) plots the computational time (in seconds) required to find GNEs within a single non-empty set $\mc{K}_{J}$. The time is benchmarked against varying problem dimensions $(n_1,n_2)$, while keeping the number of constraints proportional $(m_i = n_i+1)$. This figure highlights the computational cost associated with finding GNEs for a fixed $J\in \mc{P}$.
\item Figure~\ref{fig:time_analysis}(b) demonstrates the impact of the number of constraints (and thus, $|\mc{P}|$). Here, we fix the dimensions $n_1=n_2=2$ and plot the total computational time against an increasing number of constraints $(m_1,m_2)$.
Each data point is annotated with the corresponding cardinality $|\mc{P}|$.
\end{itemize}
\begin{figure}[ht]
    \centering 
    
    \begin{subfigure}[b]{0.49\textwidth}
        \centering
        \includegraphics[width=\linewidth]{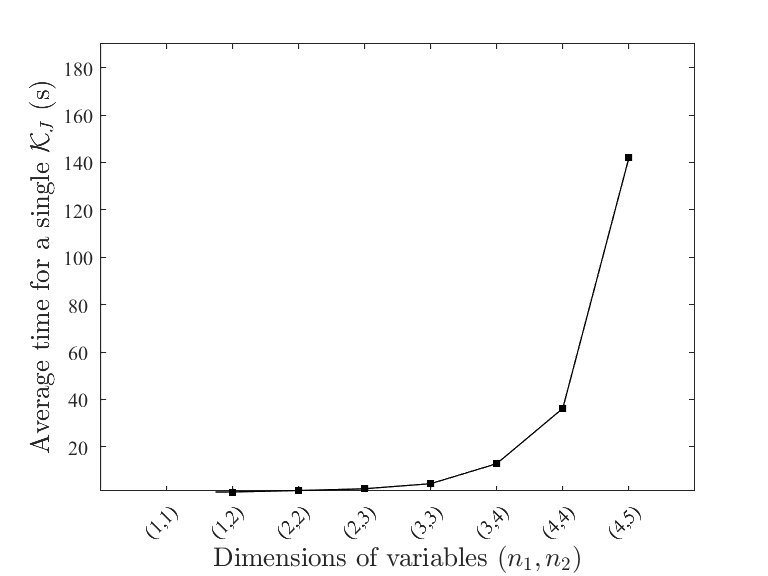}
        \caption{Average computation time for solving a single $\mathcal{K}_J$ versus variable dimensions $(n_1, n_2)$.}
        \label{fig:nvs_time}
    \end{subfigure}
    \hfill
    \begin{subfigure}[b]{0.49\textwidth}
        \centering
        \includegraphics[width=\linewidth]{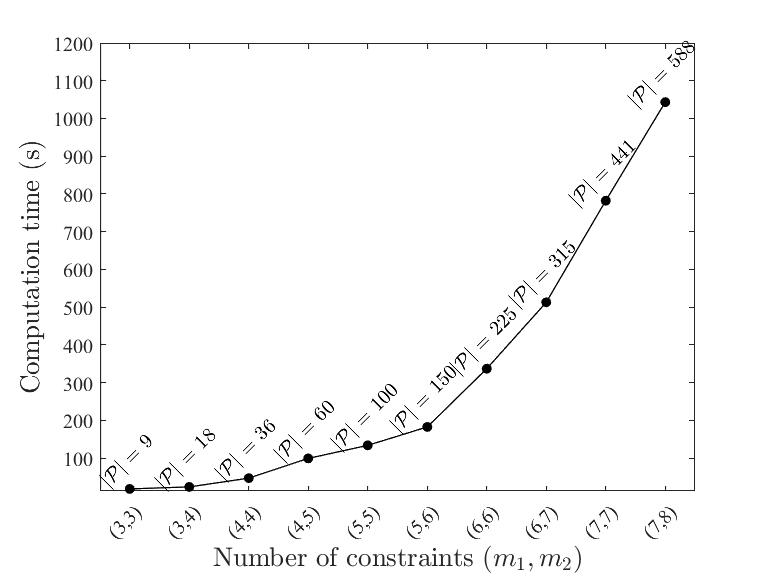}
        \caption{Total computation time versus the number of constraints $(m_1, m_2)$.}
        \label{fig:mvs_time} 
    \end{subfigure}
        \caption{The relation between the computational time and the problem sizes of GNEPs in Example~\ref{ex:random} with $N=2$.} 
    \label{fig:time_analysis}
\end{figure}

In this experiment, we observed some limitations of our method.
First, as the values of $N$, $n$, and $m$ increase, $|\mathcal{P}|$ grows rapidly.
Consequently, the overall computational time of Algorithm~\ref{alg:pLME_J} will increase quickly, although the optimization problems~(\ref{eq:findfirstKKT}), (\ref{eq:verifyGNE}), (\ref{theta>=uj}), and (\ref{eq:branch_max}) in Algorithm~\ref{alg:pLME_J} can be solved with low computational cost.
Second, our approach is sensitive to scalability.
Specifically, if we solve the optimization problem~(\ref{rand:cvxGNEP})
with a large value of $n_1 + \cdots + n_N$,
Algorithm~\ref{alg:pLME_J} may not produce accurate results due to numerical issues.
\end{example}

\begin{example}
\label{ex:lifted}
Consider the 2-player nonconvex GNEP
\begin{equation*}
\begin{array}{clccl}
\underset{x_1 \in \mathbb{R}^4}{\text{min}} & 3\|x_1\|^2 + x_{1,2} \cdot \mathbf{1}^Tx_1 & \vline & \underset{x_2 \in \mathbb{R}^4}{\text{min}} & -\|x_2\|^2 - x_{2,3} \cdot \mathbf{1}^Tx_2  \\
\st & A_1x_1\geq b_1(x_2), & \vline & \st & A_2x_2 \geq b_2(x_1),\\
\end{array}
\end{equation*}
where
\begin{gather*}
       A_1 = \left[\begin{array}{rrrr}
       0 & 1 & 1 & 3\\
       -1 & -2 & 3 & 0\\
        1 &   -1  &    1   &   0\\
        0    & -1   &   2    &  2\\
       -1    & -1     & 1      &2
       \end{array} \right],\,\,
        A_2 = \left[\begin{array}{rrrr}
        0 & 1 & 1 & 3\\
        -2 & 2 & -1 & 1\\
         0  &  -2  &   1&    -2\\
     0  &  -1    & 0    & 2\\
     2  &   1    & 0    &-2
       \end{array} \right],\\
b_1(x_2) = \bbm
       5-\|x_2\|^2 \\
       \|x_2\|^2 \\
        3x_{2,1}x_{2,2}+2x_{2,3}x_{2,4}-x_{2,2} \\
        3+x_{2,1}x_{2,3}\\
        2+x_{2,2}x_{2,4}
       \ebm,\,\,
b_2(x_1) = \bbm
       5-\|x_1\|^2\\
       \|x_1\|^2 \\
        2-2x_{1,3} +x_{1,1}x_{1,2}+x_{1,3}x_{1,4} \\
        -1 +x_{1,1}x_{1,3} \\
        -2 +x_{1,2}x_{1,4}
       \ebm.
\end{gather*}
There are 25 sets $J \in \mc{P}$.
It took around $16.05$ seconds to find all GNEs by Algorithm~\ref{alg:pLME_all}
and the computational time for each $\mc{K}_J$ is between $0.13-2.80$ seconds.
The first GNE was detected within $4.91$ seconds, which is
\begin{equation*}
    x_1^* = (-0.4085,-1.1070,1.9636,0.1845),\,\, x_2^* = (-2.0032,1.4764,1.5146,-0.1629).
\end{equation*}
This is the unique GNE and the unique KKT point. It is contained in $\mc{K}_J$ for
\[\begin{aligned}
    J \,=\,& \big(\, \{1,2,3,4\},\, \{2,3,4,5\}\,\big), \quad \big(\,\{1,2,3,5\},\, \{2,3,4,5\}\,\big), \\
    & \big(\, \{1,2,4,5\},\, \{2,3,4,5\}\, \big),\quad \mbox{or}\quad
    \big(\, \{2,3,4,5\},\, \{2,3,4,5\}\, \big).
\end{aligned}\]

We also implemented the homotopy method in \cite{Lee2023} for finding GNEs of this GNEP.
The polyhedral homotopy continuation is implemented in the Julia software {\tt homotopycontinuation.jl} \cite{homotopycontinuation}, which found $17100$ complex roots to the complex KKT system, including $1860$ real solutions. 
Among these real roots,
we got the same KKT point $(x_1^*, x_2^*)$,
which is verified to be a GNE by solving (\ref{eq:verifyGNE}).
It took around $210.88$ seconds for the polyhedral homotopy
to solve the complex KKT system and $1.50$ seconds to verify the GNE.
We would like to remark that since the number of computed complex roots of the complex KKT system is less than its mixed volume that equals $24611$, the homotopy method cannot certify the uniqueness of the computed GNE; see \cite[Theorem~3.2]{Lee2023}.

Moreover, we tested the augmented Lagrangian method in \cite{kanzow2016}
and the interior point method in \cite{dreves2011solution} for finding GNEs.
The augmented Lagrangian method cannot find a GNE after $1000$
outer iterations because the augmented Lagrangian subproblem cannot be solved accurately.
Also, the interior point method failed to find a GNE
within $1000$ iterations since the Newton directions are usually not descent directions.
\end{example}

\subsection{GNEPs from real-world applications}
GNEPs arise in a variety of real-world scenarios.
In this subsection, we present an example that illustrates how GNEPs can be used to model such practical applications.

\begin{example}
\label{ex:JI}
Consider the environmental pollution control problem for two countries under the mechanism of \textit{joint implementation} \cite{Breton}.
For each $i,j\in [2]$, $x_{i,0}$ is the (gross) emissions for the industrial outputs of country $i$, and $x_{i,j}$ represents the investment by country $i$ to the environmental projects at country $j$.
Then, the accounted-for-emission for country $i$ is $x_{i,0} - \gamma_1x_{i,1} + \gamma_2x_{i,2} $, which must be kept below a prescribed level $E_i>0$.
Also, its net emission is $R_i(x) =  x_{i,0} - \gamma_i(x_{1,i} + x_{2,i}).$
For the country $i$, assume the revenue from the industrial outputs is $x_{i,0}(b_i-0.5x_{i,0})$, the cost of investments in environmental projects is  $0.5(x_{i,i}^2+(x_{i,3-i}+x_{3-i,3-i})^2)$,
and the environmental cost incurred by pollution is quadratic in $R_i$ that equals $\sum_{j=1}^2 (c_{i,j}R_j^2 + d_{i,j}R_j)$.
Then, the optimization problem for each country is formulated as follows:
\begin{equation*}
\text{F}_i(x_{-i}) : \left\{\begin{array}{cl}
\underset{x_i \in \mathbb{R}^3}{\text{min}} & -x_{i,0}(b_i-0.5x_{i,0}) + 0.5(x_{i,i}^2+(x_{i,3-i}+x_{3-i,3-i})^2) \\
& \qquad\qquad\qquad\qquad\qquad+ \sum\limits_{j=1}^2 (c_{i,j}R_j(x)^2 + d_{i,j}R_j(x))\\
\st & b_i - x_{i,0}\geq 0,\, x_{i,1} \geq 0, \, x_{i,2} \geq 0,\\
    & E_i - (x_{i,0} - \gamma_1x_{i,1} + \gamma_2x_{i,2}) \ge 0,\\
    & x_{1,0} - \gamma_i(x_{1,1} + x_{2,1})\ge 0,\\
    & x_{2,0} - \gamma_i(x_{1,2} + x_{2,2})\ge 0,
\end{array}\right.
\end{equation*}
where parameters are set as $E_1 = 2, \, E_2 = 4,$ and
\begin{equation*}
\begin{array}{llllllll}
  b_1 = 2,& b_2 = 3,&  c_{1,1} = 0.2, & c_{1,2} = 0.3, & c_{2,1} = 0.4, & c_{2,2} = 0.2, \\
  d_{1,1} = 0.5, &  d_{1,2} = 0.75, & d_{2,1} = 0.8, & d_{2,2} = 1.2,  & \gamma_1 = 1.2, & \gamma_2 = 1 .
\end{array}
\end{equation*}
There are a total of 195 $J \in \mathcal{P}$.
It took around 164.60 seconds to find all GNEs by Algorithm~\ref{alg:pLME_all} and the computational time for each $\mathcal{K}_J$ is between 0.99-3.55 seconds.
The first GNE was detected within 3.62 seconds, which is
\begin{equation*}
    x_1^* = (1.4455, 0.6655, 0.0000),\quad
    x_2^* = (1.6667, 0.4225, 1.3333).
\end{equation*}
This is the unique GNE.
It is contained in 75 $\mathcal{K}_J$'s, and it is the only KKT point.
\end{example}

\subsection{GNEPs with shared constraints}

We consider GNEPs with shared constraints,
i.e., all players have the same constraints.
Solving this class of GNEPs is particularly hard.
This is because when there are shared constraints, the complex variety given by the KKT equations has a positive dimension (see \cite{nie2022algebraic}).
Consequently, the critical set $\mc{K}$ is often an infinite set.
For this class of GNEPs, our method may not be able to find all GNEs.
However, it is still applicable for such GNEPs that have quasi-linear constraints,
and it may still get some GNEs.
For instance, when the GNEP is convex, our method guarantees to find at least one GNE,
if there exists one; see Theorem~\ref{eq:solveone}.
The following is an example that has shared constraints.

\begin{example}
\label{ex:shared}
Consider the 2-player GNEP with shared constraints:
\begin{equation*}
\begin{array}{clccl}
\underset{x_1 \in \mathbb{R}^4}{\text{min}} & f_1(x_1,x_2)& \vline &\underset{x_2 \in \mathbb{R}^4}{\text{min}} & f_2(x_1,x_2)  \\
\st & 0\le x_{1,j}\le 1, \ (j\in [4]), & \vline &  \st &  -1\le  x_{2,j} \le 1, \ (j\in [4]), \\
& 2-\mathbf{1}^Tx_1+\mathbf{1}^Tx_2 =  0, & \vline &  & 2-\mathbf{1}^Tx_1+\mathbf{1}^Tx_2 = 0.\\
\end{array}
\end{equation*}
Here the objective functions $f_1$ and $f_2$
are the same as in Example~\ref{ex:lifted}.
In this setting, the GNEP is nonconvex.
We implement Algorithm~\ref{alg:pLME_all} and it takes around 12.17 seconds to find one GNE
\begin{equation*}
\begin{array}{ll}
x_1^* = (0, 0, 0, 0),  & x_2^* = (-1, -1, -1, 1).
\end{array}
\end{equation*}
Then, it finds a KKT point $u$ such that $u_1 = \mathbf{0}\in\re^4, u_2 = (-0.75,-0.75,0.25,-0.75)$, which is not a GNE.
At the KKT point $u$, we cannot find an appropriate $\delta$  such that $\theta_{\max} = \theta(u)$ by solving (\ref{eq:branch_max}) (see Steps~3 and 4 of Algorithm~\ref{alg:pLME_J}).
This may arise because there are infinitely many KKT points near $u$.
Therefore, though our method successfully finds a GNE, it cannot find all of them.
\end{example}

Note that the ADMM-type method \cite{Borgens2021} is applicable to GNEPs with shared constraints given by equalities.
We implement this approach with the fixed regularization (see \cite[Algorithm~4.1]{Borgens2021}) and parameters settled in the same way as in \cite{Borgens2021} for solving the GNEPs in Example~\ref{ex:shared}.
The Moment-SOS hierarchy is applied to solve each player's subproblem.
Moreover,  GNEPs with shared constraints can be reformulated as variational inequality (VI) problems $\mbox{VI}(X,\mathcal{F})$, where $X$ is the feasible set of the GNEP and $\mathcal{F}\coloneq (\nabla_{x_i} f_i)_{i=1}^N$; see \cite{Facchinei2010generalized,Facchinei2010} and references therein.
Note that the feasible sets $X_i$'s of the GNEP in Example~\ref{ex:shared} are convex.
We apply the software {\tt path} \cite{path} in MATLAB to solve the KKT system of $\mbox{VI}(X,\mathcal{F})$.
For both methods, we use three different initial points: $x^{(0)} = \mathbf{0}$, $x^{(0)} = 0.1 \cdot \mathbf{1}$, and $x^{(0)} = \mathbf{1}$ (all in $\re^8$), together with the same initial Lagrange multipliers given by all zero vectors.
The numerical results are reported in Table~\ref{tab:admmvi_comparison}.
In the table,
the `Initial point' represents the initial point for the iterative methods.
Also, `KKT' means that the method converges to a KKT point that is not a GNE, `GNE' means that the method obtains a GNE,
and `NC' means that the method does not converge.
\begin{table}[ht!]
\centering
\caption{Comparison of Algorithm~\ref{alg:pLME_all} with the ADMM and VI approachs}
\label{tab:admmvi_comparison}
\begin{tabular}{|c|c|c|c|c|c|c|c|c|}
\hline
& \multicolumn{3}{c|}{ADMM} & \multicolumn{3}{c|}{VI}  & \multicolumn{1}{c|}{Alg.~\ref{alg:pLME_all}}\\
\cline{1-8}
Initial point & \multicolumn{1}{c|}{$x^{(0)} = \mathbf{0}$} & \multicolumn{1}{c|}{$0.1 \cdot \mathbf{1}$} & $ \mathbf{1}$ &  \multicolumn{1}{c|}{$x^{(0)} = \mathbf{0}$} & \multicolumn{1}{c|}{$ 0.1 \cdot \mathbf{1}$} & $\mathbf{1}$ & - \\
\hline
Result & \multicolumn{1}{c|}{NC} & \multicolumn{1}{c|}{NC} & NC & KKT & KKT & KKT &  GNE  \\ \hline
\end{tabular}
\end{table}
For the nonconvex GNEP in Example~\ref{ex:shared}, the ADMM does not converge for all three given initial points, as the subproblem cannot be solved at some iterates.
Interestingly, if we set $x^{(0)} = - \mathbf{1}$, then the ADMM method found a GNE $x_1 = {\bf 0},\ x_2 = (1,-1,-1,-1)$.
For the VI method, it found the KKT point $x_1 = {\bf 0},\ x_2 = (-1,-1,1,-1)$ when $x^{(0)} = \mathbf{0}$, and found another KKT point $x_1 = {\bf 0},\ x_2 =  (-0.5,-1,0.5,-1)$ when $x^{(0)} = 0.1\cdot \mathbf{1} \ \mbox{and}\ \mathbf{1}$, all of which cost around 0.1 second.
However, neither of them is a GNE.

\subsection{Comparison with Gauss-Seidel method}\label{sec:GS}
In this subsection, we compare Algorithm~\ref{alg:pLME_all} with the Gauss-Seidel method \cite{Nie2020gs}.
We apply the Gauss-Seidel method to solve GNEPs in seven examples (Examples \ref{ex:0909_0324}-\ref{ex:0324_0129} and \ref{ex:lifted}-\ref{ex:shared}), using three different initial points: $x^{(0)} = \mathbf{0}$, $x^{(0)} = 0.1 \cdot \mathbf{1}$, and $x^{(0)} = \mathbf{1}$ (all with proper dimensions).
We set a maximum of $300$ loops for the Gauss-Seidel method, in which the Moment-SOS hierarchy is applied to solve each player's normalized subproblem.
The results are reported in Table~\ref{tab:gs_comparison}.
In the table, the notation \#$x^*$ indicates the number of computed GNEs, and the column “Time” reports the time (in seconds) required by Algorithm~\ref{alg:pLME_all} to find the first GNE.
For the Gauss-Seidel method, `NC' indicates that the Gauss-Seidel method does not converge, while successful cases show the corresponding execution time (in seconds).
\begin{table}[ht!]
\centering
\caption{Comparison of Algorithm~\ref{alg:pLME_all} and the Gauss-Seidel method}
\label{tab:gs_comparison}
\begin{tabular}{|c|cc|ccc|}
\hline
\multirow{2}{*}{Example} & \multicolumn{2}{c|}{Algorithm~\ref{alg:pLME_all}}  & \multicolumn{3}{c|}{Gauss-Seidel} \\
\cline{2-6}
 & \multicolumn{1}{c|}{\#$x^{*}$} & Time &  \multicolumn{1}{c|}{$x^{(0)} = \mathbf{0}$} & \multicolumn{1}{c|}{$x^{(0)} = 0.1 \cdot \mathbf{1}$} & $ x^{(0)} = \mathbf{1}$ \\
\hline
\ref{ex:0909_0324} & \multicolumn{1}{c|}{6} & 149.09 & \multicolumn{1}{c|}{NC} & \multicolumn{1}{c|}{NC} & NC \\ \hline
\ref{ex:0802_0313} & \multicolumn{1}{c|}{1} & 151.88 & \multicolumn{1}{c|}{23.47} & \multicolumn{1}{c|}{28.05} & 18.39 \\ \hline
\ref{ex:D68_0305} & \multicolumn{1}{c|}{1} & 127.56 & \multicolumn{1}{c|}{NC} & \multicolumn{1}{c|}{NC} & NC \\ \hline
\ref{ex:0324_0129} & \multicolumn{1}{c|}{2} & 116.59 & \multicolumn{1}{c|}{NC} & \multicolumn{1}{c|}{NC} & NC \\ \hline
\ref{ex:lifted} & \multicolumn{1}{c|}{1} & 4.91 & \multicolumn{1}{c|}{NC} & \multicolumn{1}{c|}{NC} & NC \\ \hline
\ref{ex:JI} & \multicolumn{1}{c|}{1} & 3.62 &\multicolumn{1}{c|}{12.88} & \multicolumn{1}{c|}{NC} & NC \\ \hline
\ref{ex:shared} & \multicolumn{1}{c|}{-} & 17.43 &\multicolumn{1}{c|}{119.13} & \multicolumn{1}{c|}{NC} & NC \\ \hline
\end{tabular}
\end{table}

The Gauss-Seidel method requires solving normalized subproblems of every player in each loop. Consequently, the infeasibility of any subproblem ${\bf F}_i(\xmi^{(l)})$ with $i\in [N]$ and some iterate $x^{(l)}$ will immediately terminate the iteration and lead to a failure of computing GNEs.
This is particularly the case for GNEPs with many constraints, which are of central interest to the pLME method.
Moreover, the performance of Gauss-Seidel method depends on the choice of the initial points, due to the nature of iterative methods.
Indeed, in our experiment, the feasible set $X_i(x^{(l)}_{-i})$ becomes empty for some $i\in [N]$ when we implement the Gauss-Seidel method to solve GNEPs in Examples~\ref{ex:0909_0324}, \ref{ex:D68_0305}, \ref{ex:0324_0129} and \ref{ex:lifted} for all choices of initial points, 
and Examples~\ref{ex:JI} and \ref{ex:shared} for all nonzero initial points.

\section{Conclusion and Discussions}
\label{sc:con}

This paper studies GNEPs with quasi-linear constraints that are defined by polynomials.
We propose a new partial Lagrange multiplier expression approach with KKT conditions.
By using partial Lagrange multiplier expressions,
we represent KKT sets of such GNEPs by a union of simpler sets with convenient expressions.
This helps to relax GNEPs into finite groups of branch polynomial optimization problems.
The latter can be solved efficiently by Moment-SOS relaxations.
Under some genericity assumptions, we develop algorithms that either find all GNEs
or detect their nonexistence.
Numerical experiments are given to show the efficiency of our method.
There is great potential for our method.
It can be interesting future work to apply our method for solving GNEPs arising from machine learning and data science applications.

We remark that GNEPs with quasi-linear constraints are typically more difficult than
GNEPs with linear constraints.
To see this, we compare their algebraic degrees,
which count numbers of complex solutions to KKT systems \cite{nie2022algebraic}.
For the convenience of our discussion, we suppose that for each $i\in[N]$,
$f_i(x)$ is a quadratic polynomial in $x$, $A_i$ is a $m_i$-by-$n_i$ matrix,
and every $b_{i,j}\, (j\in[m_i])$ is a polynomial in $\xmi$ whose degree equals $d_{i,j}$.
Without loss of generality, we also assume that $m_i\le n_i$ for each $i\in[N]$,
and all constraints are active at every KKT point (otherwise,
one may compute the algebraic degree by enumerating all active sets,
see \cite[Theorem~5.2]{nie2022algebraic}).
When all $f_i(x)$, $A_i$ and $b_i(\xmi)$ are generic, the algebraic degree is
$\prod_{i=1}^N\prod_{j=1}^{m_i} \max(1,d_{i,j}).$
In particular, when the GNEP has only linear constraints
(i.e., $d_{i,j} \le 1$ for all $i,j$),
the algebraic degree is equal to $1$, which is much less than that
for general cases of quasi-linear constraints
(i.e., $d_{i,j}$ are greater than $1$).
For instance, when $N=2$, $m_1=m_2=4$, and $d_{i,j}=2$ for all $i,j$,
the algebraic degree for the GNEP with quasi-linear constraints is $2^8 = 256$,
under some genericity assumptions.
We refer to \cite{nie2022algebraic} for more details about algebraic degrees.

Moreover, we note that our method aims at finding GNEs and uses the KKT conditions to compute candidate solutions.
When the GNEP is convex, GNEs are equivalent to the KKT points; thus, our method finds GNEs directly.
However, for nonconvex GNEPs that are of the primary interest of this paper, we need to select GNEs from the KKT points by solving polynomial optimization problems.
Note that for nonconvex GNEPs, all players' nonconvex optimization subproblems are globally minimized at a GNE.
There exist approaches characterizing GNEs directly by some equivalent reformulations for GNEPs, such as the Nikaido-Isoda function reformulation (see \cite{Facchinei2010}), and the reformulation that uses Putinar's Positivstellensatz in \cite{Couzoudis2013}.
However, the reformulations are usually extremely difficult to solve in computational practice when the GNEP is nonconvex.
Also, the Gauss-Seidel method can be viewed as an approach that computes GNEs directly.
Nonetheless, its convergence is not guaranteed in general.
It is interesting future work to find computationally efficient approximations for the set of GNEs that are better than the set of KKT points.

Last, we acknowledge that our proposed method may not be efficient for solving large scale GNEPs. This is because it requires solving semidefinite programming relaxations,
whose sizes grow fast when the number of variables increases.
Moreover, though the instances of $|\mc{P}|$ are much smaller than the number of all possible active sets (as shown in Table~\ref{table:numberP}), typically one still needs to enumerate a large number of sets $J$ when the number of constraints $m_i$ is relatively large (say, more than $20$).
We refer to Example~\ref{ex:random} for the numerical performance of our method with increasing values of $n_i$ and $m_i$. Generally, solving large scale GNEPs is computationally challenging.
This is important future work.

\subsection*{Acknowledgements}

This project was begun at a SQuaRE at the American Institute for Mathematics.
The authors thank AIM for providing a supportive and mathematically rich environment.
Jiawang Nie is partially supported by the NSF grant DMS-2110780 and DMS-2513254.
Xindong Tang is partially supported by the Hong Kong Research Grants Council HKBU-15303423.

\appendix
\section{GNEPs from existing references}
\label{sc:existing_refs}
\begin{example} (FKA3 \cite{FacKan10}).
\label{ex:A03}
Consider the 3-player GNEP
\begin{equation}\label{eq:convGNEP}
\text{F}_i(x_{-i}):\
\left\{\begin{array}{cl}
\min\limits_{x_i \in \mathbb{R}^{n_i}} &  \frac{1}{2} x_i^T C_i x_i + x_i^T (D_i x_{-i}+t_i)\\
\st & A_i x_i \geq b_i(x_{-i}),
\end{array}\right.
\end{equation}
where $n_1 = 3,\, n_2 = n_3 = 2$ and
\begin{equation*} \begin{gathered}
C_1 =\left[\begin{array}{rrr}20 & 5 & 3 \\ 5 & 5 & -5 \\ 3 & -5 & 15 \end{array}\right],\,\,\,\,
C_2 = \left[\begin{array}{rr} 11 & -1 \\ -1 & 9\end{array}\right],\,\,\,\,
C_3 = \left[\begin{array}{rr} 48 & 39 \\ 39 & 53\end{array}\right], \\
D_1 = \left[\begin{array}{rrrr} -6 & 10 & 11 & 20 \\ 10 & -4 & -17 & 9 \\ 15 & 8 & -22 & 21 \end{array}\right], \,\,\,\,
D_2 = \left[\begin{array}{rrrrr} 20 & 1 & -3 & 12 & 1 \\ 10 & -4 & 8 & 16 & 21 \end{array}\right], \\
D_3 = \left[\begin{array}{rrrrr} 10 & -2 & 22 & 12 & 16 \\ 9 & 19 & 21 & -4 & 20 \end{array}\right],\,\,\,\,
t_1 = \left[\begin{array}{r} 1\\ -1 \\ 1 \end{array}\right],\,\,\,\,
t_2 = \left[\begin{array}{r} 1\\ 0\end{array}\right],\,\,\,\,
t_3 = \left[\begin{array}{r} -1\\ 2 \end{array}\right].
\end{gathered}
\end{equation*}
The following are the constraints for each player.
\begin{equation*}\begin{aligned}
    \mbox{1st player}:\, &\left\{\begin{array}{l}
    -10 \cdot \mathbf{1} \leq x_1 \leq 10 \cdot \mathbf{1},\,x_{1,1}+x_{1,2}+x_{1,3} \leq 20,\\
    x_{1,1}+x_{1,2}-x_{1,3} \leq x_{2,1} - x_{3,2} +5,
    \end{array}
    \right.\\
\mbox{2nd player}:\, &
    -10 \cdot \mathbf{1} \leq x_2 \leq 10 \cdot \mathbf{1},\,x_{2,1} - x_{2,2} \leq x_{1,2} + x_{1,3} - x_{3,1} +7,\\
\mbox{3rd player}: \, &
    -10 \cdot \mathbf{1} \leq x_3 \leq 10 \cdot \mathbf{1},\,x_{3,2} \leq x_{1,1} + x_{1,3} -x_{2,1} +4.
\end{aligned}
\end{equation*}
\end{example}

\begin{example} (FKA4 \cite{FacKan10}).
\label{ex:A04}
Consider (\ref{eq:convGNEP}) with $N = 3, \,n_1 = 3,\, n_2 = n_3 = 2,$
\begin{equation*}
\begin{gathered}
C_1 = \left[\begin{array}{ccc} 20 + x_{2,1}^2 & 5 & 3 \\ 5 & 5 + x_{2,2}^2 & -5 \\ 3 & -5 & 15 \end{array}\right],\ 
C_2 = \left[\begin{array}{cc} 11 + x_{3,1}^2 & -1 \\ -1 & 9\end{array}\right],\\
C_3 = \left[\begin{array}{cc} 48 & 39 \\ 39 & 53 +x_{1,1}^2\end{array}\right], 
\end{gathered}
\end{equation*}
and $D_i$ and $t_i$ are the same as Example \ref{ex:A03}.
The following are the constraints:
\[\begin{aligned}
    \mbox{1st player}:\, & \left\{\begin{array}{l}
    \mathbf{1} \leq x_1 \leq 10 \cdot \mathbf{1},\,x_{1,1}+x_{1,2}+x_{1,3} \leq 20,\\
    x_{1,1}+x_{1,2}-x_{1,3} \leq x_{2,1} - x_{3,2} +5,
    \end{array}
    \right.\\
\mbox{2nd player}:\, & \mathbf{1} \leq x_2 \leq 10 \cdot \mathbf{1},\,x_{2,1} - x_{2,2} \leq x_{1,2} + x_{1,3} - x_{3,1} +7,\\
\mbox{3rd player}:\, &  \mathbf{1} \leq x_3 \leq 10 \cdot \mathbf{1},\,x_{3,2} \leq x_{1,1} + x_{1,3} -x_{2,1} +4.
\end{aligned}\]
\end{example}

\begin{example} (FKA8 \cite{ facchinei2009generalized,FacKan10}).
\label{ex:A08}
Consider the 3-player GNEP
\begin{equation*}
{\small \begin{array}{clcclccl}
\underset{x_1 \in \mathbb{R}^1}{\min} & -x_1 & \vline & \underset{x_2 \in \mathbb{R}^1}{\text{min}} & (x_2-0.5)^2 & \vline & \underset{x_3 \in \mathbb{R}^1}{\text{min}} & (x_3-1.5x_1)^2 \\
\st & x_3 \leq x_1+x_2 \leq 1, & \vline &\st & x_3 \leq x_1+x_2 \leq 1, & \vline &\st & 0 \leq x_3 \leq 2, \\ & 0 \leq 2x_1 \leq x_3, & \vline & & x_2 \geq 0, & \vline & &  2x_3 \geq x_1 + 2x_2.
\end{array}}
\end{equation*}
The original problem in \cite{facchinei2009generalized} has infinitely many KKT points,
so we added extra constraints to the first and third players' optimization
so that the KKT set is finite.
\end{example}

\begin{example} (NT59 \cite{Nie2020nash}).
\label{ex:F59}
Consider the environmental pollution control problem for $N = 3$ countries in the case \textit{autarky} \cite{Breton}.
Let $x_{i,1}\,(i = 1,2,3)$ denote the (gross) emissions from country $i$.
The revenue of each country depends on its own emissions.
The variable $x_{i,2}$ represents the investment by country $i$ in local environmental projects.
Set these variables together with net emissions and damage costs.
The optimization problem for each country is formulated as follows:
\begin{equation*}
\text{F}_i(x_{-i}) : \left\{\begin{array}{cl}
\underset{x_i \in \mathbb{R}^2}{\text{min}} &  -x_{i,1}\left(b_i - \frac{1}{2} x_{i,1} \right) + \frac{x_{i,2}^2}{2} + d_i (x_{i,1}-\gamma_i x_{i,2}) + \sum\limits_{j \ne i} c_{i,j} x_{i,2} x_{j,1}\\
\st & x_{i,2} \geq 0, \, x_{i,1} \leq b_i,\, 0 \leq x_{i,1}- \gamma_i x_{i,2} \leq E_i,
\end{array}\right.
\end{equation*}
where the parameters are set as
\begin{equation*}
\begin{array}{llllll}
  b_1 = 1.5,& b_2 = 2,& b_3 = 1.8, & c_{1,2} = 0.2, & c_{1,3} = 0.3, & c_{2,1} = 0.4, \\
  c_{2,3} = 0.2, & c_{3,1} = 0.5, & c_{3,2} = 0.1, & d_1 = 0.8, & d_2 = 1.2, & d_3 = 1.0, \\
  E_1 = 3, & E_2 = 4, & E_3 = 2, & \gamma_1 = 0.7, & \gamma_2 = 0.5, & \gamma_3  = 0.9.
\end{array}
\end{equation*}
\end{example}

\begin{example} (NT510 \cite{Nie2020nash}).
\label{ex:F510}
Consider the electricity market problem in \cite{FacKan10} with $N=3$ generating companies.
Each company $i$ possesses $i$ generating units, and the power generation of its $j$th generating unit is denoted by $x_{i,j}$.
Taking into account the maximum generation capacity, the operating cost of each unit, and the electricity price, each company aims to maximize its own profit.
Under this setting, the optimization problem for company $i$ can be formulated as follows:
\begin{equation*}
\text{F}_i(x_{-i}): \left\{\begin{array}{cl}
\min\limits_{x_i \in \mathbb{R}^i} & \sum\limits_{j=1}^i \left( \frac{1}{2}c_{i,j}  x_{i,j}^2 - d_{i,j}  x_{i,j} \right) - (10 - \mathbf{1}^Tx) \mathbf{1}^T x_i  \\
\st & 0 \leq x_{i,j} \leq E_{i,j} \,\, (\forall j \in [i]),
\end{array}\right.
\end{equation*}
where the parameters are set as
\begin{equation*}
\begin{array}{llllll}
  c_{1,1} = 0.4,& c_{2,1} = 0.35,& c_{2,2} = 0.35, & c_{3,1} = 0.46, & c_{3,2} = 0.5, & c_{3,3} = 0.5, \\
  d_{1,1} = 2, & d_{2,1} = 1.25, & d_{2,2} = 1, & d_{3,1} = 2.25, & d_{3,2} = 3, & d_{3,3} = 3, \\
  E_{1,1} = 2, & E_{2,1} = 2.5, & E_{2,2} = 0.67, & E_{3,1} = 1.2, & E_{3,2} = 1.8, & E_{3,3}  = 1.6.
\end{array}
\end{equation*}
\end{example}

\begin{example}(NTGS53 \cite{Nie2020gs}).
\label{ex:GS53}
Consider the 2-player GNEP
\begin{equation*}
{\small \begin{array}{clccl}
\underset{x_1 \in \mathbb{R}^2}{\text{min}} & x_{1,1}(x_{1,2} + 2x_{2,1} +2x_{2,2}) & \vline & \underset{x_2 \in \mathbb{R}^2}{\text{min}} & \|x_1\|^2 - \|x_2\|^2 \\
& +x_{1,2}(x_{2,1}+x_{2,2}) + 2x_{2,1}x_{2,2} & \vline & & \\
\st & \mathbf{1}^T x = 1,\,\, x_1 \geq 0,\,\, 2 \cdot \mathbf{1}^Tx_1 \geq 1, & \vline & \st & \mathbf{1}^T x = 1,\,\, x_2 \geq 0, \,\, \mathbf{1}^Tx_2 \geq \mathbf{1}^Tx_1.
\end{array}}
\end{equation*}
The original problem in \cite{Nie2020gs} has infinitely many KKT points,
so we added extra constraints to each player's optimization
so that the KKT set is finite.
\end{example}

\begin{example}(NTGS54 \cite{Nie2020gs}).
\label{ex:GS54}
Consider the 2-player GNEP
\begin{equation*}
\begin{array}{clccl}
\underset{x_1 \in \mathbb{R}^2}{\text{min}} & -2x_{1,2}^2 + x_{2,1}x_{1,2} + x_{1,1}x_{2,1} & \vline & \underset{x_2 \in \mathbb{R}^2}{\text{min}} & \|x_2\|^2 -2x_{2,2} \cdot \mathbf{1}^Tx_1 \\
\st & \mathbf{1}^T x = 1,\,\, x_{1,1} \geq 0.1,\,\, x_{1,2} \geq 0.1, & \vline &\st & \mathbf{1}^T x = 1,\,\, x_{2} \geq 0.1,  \\
 & x_1 \geq x_2,  & \vline&  & x_2 \geq x_1, \,\, 0.2 \geq x_{1,1} + x_{2,1}.
\end{array}
\end{equation*}
The original problem in \cite{Nie2020gs} has infinitely many KKT points,
so we added extra constraints to each player's optimization
so that the KKT set is finite.
\end{example}

\begin{example} (FR33 \cite{Feinstein2023}).
\label{ex:FR33}
Consider the 2-player GNEP
\begin{equation*}
\text{F}_i(x_{-i}) : \left\{\begin{array}{cl}
\underset{x_i \in \mathbb{R}^2}{\text{min}} &  -2x_{i,1} - (2i-1)x_{i,2}\\
\st & x_i \in X_i(x_{-i}),
\end{array}\right.
\end{equation*}
where the constraining sets are
\begin{equation*}
\begin{gathered}
    X_1(x_2) = \left\{x_1 \in \mathbb{R}^2 \left| \begin{array}{l}
    0 \leq x_{1,1} \leq 5, \,\, 0 \leq x_{1,2} \leq 2.5,\,\, x_{1,1}+2x_{1,2} \leq 5,\\
    4x_{1,1}+x_{1,2}-\frac{16}{3}x_{2,1}-\frac{1}{3}x_{2,2}\leq 0.
    \end{array} \right.\right\},\\
    X_2(x_1) = \left\{x_2 \in \mathbb{R}^2 \left| \begin{array}{l}
    0 \leq x_{2,1} \leq 1.5, \,\, 0 \leq x_{2,2} \leq 6,\,\, 4x_{2,1}+x_{2,2} \leq 6,\\
15x_{1,1}-10x_{1,2}+x_{2,1}+2x_{2,2}\leq 0.
    \end{array} \right.\right\}.
\end{gathered}
\end{equation*}
\end{example}

\begin{example} (SAG41 \cite{Sagratella2019}). Consider the 2-player NEP
\label{ex:SAG41}
\begin{equation*}
\text{F}_i(x_{-i}) : \left\{\begin{array}{cl}
\underset{x_i \in \mathbb{R}^2}{\text{min}} &  4x_{i,1}^2 +(-1)^{i+1}2x_{1,1}x_{2,1}-\alpha_i x_{i,1} + \beta_i x_{i,2}\\
\st & x_{i,1} - x_{i,2} \leq 0,\,\, 0 \leq x_{i,1} \leq 1,\,\, 0 \leq x_{i,2} \leq 1,
\end{array}\right.
\end{equation*}
where $\alpha_1 = 10, \alpha_2 = 8, \beta_1 = 5$, and $\beta_2 = 7$.
\end{example}

\end{document}